\documentclass[11pt,a4paper]{amsart}
\usepackage{amsmath,amsfonts,amssymb,amscd,amsthm,amsrefs}

\usepackage[a4paper]{geometry}
\usepackage[shortlabels]{enumitem}

\usepackage{color, soul}
\definecolor{gr}{rgb}{0.7, 1, 0.7}
\definecolor{rr}{rgb}{1, 0.7, 0.7}

\usepackage{latexsym}

\usepackage{graphicx}

\theoremstyle{plain} 
\newtheorem{theorem}{Theorem}[section]

\newtheorem{lemma}[theorem]{Lemma}
\newtheorem{corollary}[theorem]{Corollary}

\newtheorem{proposition}[theorem]{Proposition}

\theoremstyle{definition} 
\newtheorem{definition}[theorem]{Definition}
\theoremstyle{remark} 
\newtheorem{remark}[theorem]{Remark}

\renewcommand{\mathfrak}{\mathbf}

\newcommand{\ignore}[1]{}

\newcommand{\bbC}{\mathbb{C}}
\newcommand{\bbP}{\mathbb{P}}

\newcommand{\bbN}{\mathbb{N}}

\newcommand{\bbZ}{\mathbb{Z}}

\newcommand{\bfx}{\mathbf{x}}

\newcommand{\bfe}{\mathbf{e}}

\newcommand{\bfc}{\mathbf{c}}

\newcommand{\bfz}{\mathbf{z}}

\newcommand{\bfw}{\mathbf{w}}

\newcommand{\bfv}{\mathbf{v}}

\newcommand{\cP}{\mathcal P}

\newcommand{\ddc}{{dd^c}}

\newcommand{\supp}{{\rm supp}}

\newcommand{\bif}{{\rm Bif}}

\newcommand{\crit}{{\rm Crit}}

\newcommand{\Pb}{\mathbb{P}}
\newcommand{\Cb}{\mathbb{C}}

\newcommand{\Zb}{\mathbb{Z}}

\newcommand{\Qb}{\mathbb{Q}}

\title[]{Independence of multipliers in several variables complex dynamics}

\author{Igors Gorbovickis}
\thanks{The research of the first author was supported by the German Research Foundation (DFG, project number 455038303).}
\address{Constructor University, Bremen, Campus Ring 1, 28759 Bremen, Germany}
\email{igorbovickis@constructor.university}

\author{Johan Taflin}
\thanks{The research of the second author was supported by the EIPHI Graduate School, ANR-17-EURE-0002.}
\address{Institut de Math\'ematiques de Bourgogne, UMR 5584 CNRS, Universit\'e de Bourgogne, F-21000 Dijon, France
}
\email{johan.taflin@u-bourgogne.fr}

\subjclass[2010]{}
\date{\today}

\begin{document}
\begin{abstract}
We establish the independence of multipliers for polynomial endomorphisms of $\Cb^n$ and endomorphisms of $\Pb^n.$ This allows us to extend results about the bifurcation measure and the critical height obtained in \cite{gauthier2023sparsity} to the case of regular polynomial endomorphisms of $\Cb^n$ for $n\geq3.$

An important step in the proof is the irreducibility of the spaces of endomorphisms with $N$ marked periodic points, which is of independent interest.
\end{abstract}
\maketitle

\section{Introduction, statement of the main results}

The study of algebraic families of holomorphic dynamical systems on $\Pb^1$ has recently seen an explosion of remarkable and diverse results (see e.g. \cite{DKY1}, \cite{demarco-krieger-hexi-quadratic}, \cite{favre-gauthier-DAO}, \cite{arfeux-kiwi}, \cite{ji2023daocurves}, \cite{ji2024multiplier}).

In higher dimensions, the field is still emergent, although there already exist substantial results (see \cite{gv-northcott}, \cite{gauthier2023sparsity} or \cite{zhang2024arithmeticpropertiesfamiliesplane}). These results often rely on arithmetic techniques. A significant difficulty in this framework is the lack of knowledge about the moduli spaces associated with these algebraic dynamical systems. The results of this article contribute to filling these gaps in the case of the moduli space of endomorphisms of $\Pb^n$ or regular polynomial endomorphisms of $\Cb^n$. Roughly speaking, we show that the multipliers of periodic points give local coordinates on a Zariski open subset of these moduli spaces. This could be useful for understanding bifurcations in these spaces or their arithmetic properties. 
In particular, the results above allow us to extend the results of~\cite{gauthier2023sparsity} on the moduli space of regular polynomial endomorphisms of $\Cb^2$ to the general case of $\Cb^n$ with $n\geq3$ (see Corollary \ref{cor-mu-interior} and Corollary \ref{cor-unif}).

\subsection{Independence of the eigenvalues}

For a pair of integers $d\ge 2$ and $n\ge 1$, let $\mathcal P_d^n$ be the space of all polynomials $F\colon\bbC^n\to\bbC^n$ of degree $d$ which are \emph{regular}, i.e., that extend holomorphically to the projective space $\Pb^n.$ As this last condition is algebraic, $\mathcal P_d^n$ is naturally an affine variety.

Let $F\in\mathcal P_d^n$ be a polynomial and let $\bfw\in\bbC^n$ be one of its periodic points of period $p$. Let $\lambda\in\bbC$ be an eigenvalue of $\bfw$ (i.e., an eigenvalue of the Jacobian matrix $DF^p$ at $\bfw$). If $\lambda\neq 1$, and the Jacobian matrix $D_{\bfw} F^p$ does not have other eigenvalues equal to $\lambda$ or to $1$, then it follows from the Implicit Function Theorem that the periodic point $\bfw$ and its eigenvalue $\lambda$ can be followed locally and analytically in $\mathcal P_d^n$. This gives a \textit{local eigenvalue function} of period $p$. Furthermore, analytic continuation of the eigenvalue $\lambda$ (as well as the periodic point $\bfw$) is well-defined over the whole space $\mathcal P_d^n$ and results in a (multiple valued) algebraic function. We will call it an \textit{eigenvalue function} of period $p$.

We would have liked to consider the moduli space $\mathcal{P}_d^n / \mathrm{Aff}(\mathbb{C}^n)$, where the affine group acts by conjugation, but there is a slight technical difficulty to consider it as an algebraic variety. Since the group $\mathrm{Aff}(\mathbb{C}^n)$ is not reductive, the standard results on GIT quotients do not apply. Nevertheless, in the case of endomorphisms of projective space $\mathbb{P}^n$, the group $\mathrm{PGL}_{n+1}(\mathbb{C})$ is reductive, and Petsche-Szpiro-Tepper showed in \cite{szpiro} that its action on the space $\mathrm{End}_d(\mathbb{P}^n)$ of degree $d$ endomorphisms of $\mathbb{P}^n$ has finite stabilizers. Hence, the quotient by conjugation $\mathcal{M}_d^n := \mathrm{End}_d(\mathbb{P}^n) / \mathrm{PGL}_{n+1}(\mathbb{C})$ is geometric and is an affine variety \cite[Proposition 10]{szpiro}. Therefore, in what follows, we consider the image $\tilde{\mathcal{P}}_d^n$ of $\mathcal{P}_d^n$ in $\mathcal{M}_d^n$ that we will abusively refer to as the \emph{moduli space of degree $d$ regular endomorphisms of $\Cb^n$}. There is a natural projection from $\mathcal{P}_d^n / \mathrm{Aff}(\mathbb{C}^n)$, viewed as a set, onto $\tilde{\mathcal{P}}_d^n$, which is generically injective. However, maps with totally invariant hyperplanes could be identified in $\tilde{\mathcal{P}}_d^n$, even though they are not necessarily affinely conjugated.

Since the eigenvalues of the Jacobian matrix remain unchanged under holomorphic conjugacies, it follows that the eigenvalue functions project to (multiple valued) algebraic functions on $\tilde{\mathcal P}_d^n$. The number of coefficients in a polynomial map $F\colon\bbC^n\to\bbC^n$ of degree $d$ is $n {{d+n}\choose n}$; the dimension of the moduli space $\tilde{\mathcal P}_d^n$ is 
$$
nN_{d,n} := n \left[{{d+n}\choose n}-n-1\right].
$$

Our first main result is the following: 

\begin{theorem}\label{main_theorem_1}
	Let $d$ and $n$ be any pair of integers, such that $d\ge 2$ and $n\ge 1$. Then any $nN_{d,n}$ distinct eigenvalue functions defined on $\mathcal P_d^n$ and corresponding to distinct cycles of periods not smaller than $4$, are algebraically independent over $\bbC$. 
\end{theorem}

The conclusion of the theorem means that the considered eigenvalue functions do not satisfy any non-trivial polynomial relation with complex coefficients (neither locally in an open subset of $\mathcal P_d^n$, nor globally). In other words, if $p\geq4$ then there exists a dense Zariski open subset $U_p$ of  $\tilde{\mathcal P}_d^n$ such that for any class $[f]\in U_p$, any $nN_{d,n}$ local eigenvalue functions corresponding to different periodic orbits of $f$ of period between $4$ and $p$, define local coordinates near $[f].$

We note that the statement of Theorem~\ref{main_theorem_1} is sharp in terms of the number of independent eigenvalue functions: any collection of $nN_{d,n}+1$ eigenvalue functions on $\mathcal P_d^n$ is algebraically dependent over $\bbC$, since in this case the number of the eigenvalue functions exceeds the dimension of the moduli space $\tilde{\mathcal P}_d^n$.

On the other hand, the condition on the periods of the cycles in Theorem~\ref{main_theorem_1} is far from being optimal. Using the same methods as in this paper, it is not very difficult to strengthen the result of Theorem~\ref{main_theorem_1}, allowing at least some periods to be smaller than $4$. In particular, for $n=1$ it was shown by the first author in~\cite{gorbovickis-poly} that the result of Theorem~\ref{main_theorem_1} holds without any assumptions on the periods of the cycles. However, in Theorem~\ref{main_theorem_1} we decided to sacrifice the generality of the result in favor of a more concise and transparent proof.

Observe that the result also holds in the case of the space $\mathrm{End}_d(\bbP^n)$ of degree $d$ endomorphisms of $\bbP^n.$ As already mentioned, by \cite{szpiro} the associate moduli space $\mathcal M_d^n=\mathrm{End}_d(\bbP^n)/\mathrm{PGL}_{n+1}(\bbC)$ is an affine variety of dimension $(n+1)N_{d,n}.$
\begin{theorem}\label{main_theorem_Pn}
	Let $d$ and $n$ be any pair of integers, such that $d\ge 2$ and $n\ge 1$. Define
	$$
	\tilde p:= \begin{cases}
		5 & \text{if } d=2\text{ and } n=2\\
		4 & \text{otherwise}.
	\end{cases}
	$$
	Then any $(n+1)N_{d,n}$ distinct eigenvalue functions defined on $\mathrm{End}_d(\bbP^n)$ and corresponding to distinct cycles of periods not smaller than $\tilde p$, are algebraically independent over $\bbC$.
	\end{theorem}
	
\subsection{Previous results and consequences}
When $n=1,$ the stronger versions of these results have been proven by the first author in \cite{gorbovickis-poly,gorbovickis-rat}. They are related to the seminal work of McMullen \cite{mcmullen-root} which in particular establishes that the set of all eigenvalues of periodic points of $f\colon\bbP^1\to\bbP^1$ determines the conjugacy class of $f$ up to finitely many choices in $\mathcal M_d^1$ except when $f$ is a flexible Latt\`es map. We also refer to \cite{JX-1,ji2024multiplier} for two striking new developments when $n=1$ by Ji-Xie. In higher dimension, Gauthier-Taflin-Vigny proved in \cite{gauthier2023sparsity} that all the eigenvalues of (almost) all periodic cycles determine $f\in\mathcal M_d^n$ up to finitely many choices, except when $f$ belongs to a proper algebraic subset of $\mathcal M_d^n.$ Observe that this subset, which contains several analogs of flexible Latt\`es maps, is widely unknown. It is related to the generalization in higher dimension of the characterization of algebraic stable (i.e. without bifurcations) families of rational maps on $\Pb^1$ by McMullen \cite{mcmullen-root} and of the non-ampleness locus of the critical height by Ingram \cite{ingram-height}.

The independence result of \cite{gauthier2023sparsity} also holds in the polynomial case but only in dimension $2,$ i.e., in $\tilde{\mathcal P}_d^2.$ This independence is actually a key ingredient in the proof of \cite[Theorem C]{gauthier2023sparsity} which was one of the main steps to obtain the non-Zariski density of postcritically finite (PCF) endomorphisms in $\mathcal P_d^2$ and in $\mathrm{End}_d(\bbP^n)$, as soon as $d\geq2$ and $n\geq2,$ as it has been conjectured by Ingram-Ramadas-Silverman in \cite{IRS}. Thanks to Theorem \ref{main_theorem_1}, we can now extend \cite[Theorem C]{gauthier2023sparsity} to $\tilde{\mathcal P}_d^n$ for $n\geq3,$ which was the original motivation of the present article.
\begin{corollary}\label{cor-mu-interior}
	Fix two integers $d\geq2$ and $n\geq3.$ There exists a non-empty open subset $\Omega\subset\tilde{\mathcal P}_d^n$ such that
	\begin{itemize}
		\item the open set $\Omega$ is contained in the support of the bifurcation measure $\mu_\bif,$
		\item the open set $\Omega$ contains no conjugacy class of PCF endomorphism.
	\end{itemize}
\end{corollary}
We refer to \cite{gauthier2023sparsity} for the definitions of PCF endomorphisms and the construction of the bifurcation measure is described in Section~\ref{sec_coro}. Notice that the latter was introduced by Bassanelli-Berteloot in \cite{BB1}. But, to the best of our knowledge, the simple fact that $\mu_\bif$ doesn't vanish in $\tilde{\mathcal P}_d^n$ when $n\geq3,$ which is now a direct consequence of Corollary \ref{cor-mu-interior}, was not known. Another consequence of Corollary \ref{cor-mu-interior} is that the assumptions of \cite[Theorem 7.2]{gauthier2023sparsity} are also satisfied when $n\geq3$ which gives the following uniform result, much stronger than the non-Zariski density of PCF maps in $\mathcal P_d^n$ (see \cite[Theorem D]{gauthier2023sparsity} for $n=2$).
\begin{corollary}\label{cor-unif}
Let $n\geq3$ and $d\geq2.$ There exists a dense Zariski open subset $U$ of $\mathcal P_d^n$ and an integer $B\geq1$ such that for all $f\in U$, there exists an algebraic subset $W_f$ of $\Cb^n$ of codimension $2$ with $\deg(W_f)\leq B$ which contains all the critical preperiodic points of $f$ in $\Cb^n,$ i.e.
$$\mathrm{Preper}(f)\cap\crit_f\subset W_f.$$
\end{corollary}

Here, $\crit_f$ is the critical set of $f$ in $\Cb^n$ and $\mathrm{Preper}(f)$ is its set of preperiodic points, i.e. $\mathrm{Preper}(f):=\{\bfw\in\Cb^n\,;\,\exists p>q\geq0,\ f^p(\bfw)=f^q(\bfw)\}.$ Observe that when $f$ is PCF then preperiodic points are Zariski dense in $\crit_f$. 
When $n=2,$ the sets $W_f$ are finite so the bound on $\deg(W_f)$ gives a satisfactory uniform bound on the cardinality of $W_f$ in the spirit of the uniform results in arithmetic geometry (see e.g. \cite{Dimitrov-Gao-Habegger} and \cite{DKY1}) and arithmetic dynamics (see e.g. \cite{Mavraki_Schmidt} and \cite{DeMarco-Mavraki}). In higher dimensions, although less precise, the proof of Corollary \ref{cor-unif} addresses the analogue in $\tilde{\mathcal P}_d^n$
 of Problem 6.3.9 of Yuan-Zhang \cite{YZ-adelic} about the arithmetic bigness of the adelic line bundle associated to the critical height. Observe that  \cite[Problem 6.3.9]{YZ-adelic} is widely open on $\mathcal M_d^n.$


Finally, Theorem~\ref{main_theorem_Pn} provides a positive answer to the first part of Question~19.4 of J.~Doyle and J.~Silverman~\cite{doyle-silverman} with an explicit (but possibly still non-optimal) constant~$\tilde p$ from Theorem~\ref{main_theorem_Pn}. To state the result, we observe that for any $n, p\ge 1$ and $d\ge 2$, one can consider a single valued function 
$$
\mu_{p,n,d}\colon\mathcal M_d^n \to \bbC^{n\nu(p,n,d)},
$$
obtained by taking all eigenvalue functions of all cycles of period $n$ (counted with multiplicities) and factoring them through the corresponding symmetric polynomials. (Here, $\nu(p,n,d)$ is the number of $p$-cycles of a generic map $f$ from $\mathrm{End}_d(\bbP^n)$.)

\begin{corollary}\label{cor-Doyle_Silverman}
	Let $\tilde p$ be the same as in Theorem~\ref{main_theorem_Pn}. Then for any $n\ge 1$, $d\ge 2$ and $p\ge \tilde p$, the map $\mu_{p,n,d}$ is quasi-finite on a nonempty Zariski open subset.	
\end{corollary}
Due to Theorem~\ref{main_theorem_1}, an analogous statement also holds on the moduli space $\tilde{\mathcal P}_d^n$ of polynomial maps. The proof of Corollary~\ref{cor-Doyle_Silverman} is provided in the end of Section~\ref{sec_main_proofs}. 
Several results similar to Corollary~\ref{cor-Doyle_Silverman}, but for very specific subsets of $\mathcal M_d^n$ were previously obtained in~\cite{hutz}.

We also note that the results of~\cite{gauthier2023sparsity} imply the existence of a positive integer $\hat p(n,d)$, such that for all $n\ge 1$, $d\ge 2$ and $p\ge \hat p(n,d)$, the direct product map 
$$\prod_{p\le \hat p(n,d)} \mu_{p,n,d}\colon\mathcal M_d^n\to \prod_{p\le \hat p(n,d)} \bbC^{n\nu(p,n,d)}
$$
is quasi-finite on a nonempty Zariski open subset. However, the approach of~\cite{gauthier2023sparsity} does not provide explicit bounds on $\hat p(n,d)$ nor its dependence on $n$ and $d$. Observe that when $n=1,$ Ji-Xie recently showed \cite{ji2024multiplier} that $(\mu_{1,1,d},\ldots,\mu_{p(d),1,d})$ is generically injective in $\mathcal M_d^1$ for $p(d)$ large enough, answering a question of McMullen in  \cite{mcmullen-root}. In the polynomial case in one variable, Huguin further refined this result in \cite{huguin2024modulispacespolynomialmaps}, proving that $(\mu_{1,1,d},\mu_{2,1,d})$ is already generically injective in $\tilde{\mathcal P}_d^1.$

\subsection{Irreducibility of varieties of marked periodic orbits over $\mathcal P_d^n$ and $\mathrm{End}_d(\Pb^n)$}
Our second main result that is also a key ingredient of the proof of Theorem~\ref{main_theorem_1} and Theorem~\ref{main_theorem_Pn}, is a theorem on irreducibility of the space $\cP_d^n$ or $\mathrm{End}_d(\Pb^n)$ with $N$ marked periodic orbits. In what follows, we mainly focus on the polynomial setting, i.e., in $\mathcal P_d^n$.

Let $n\geq2,$ $d\geq2$ and $N\geq1$ and let $\mathbf p=(p_1,\ldots,p_N)\in\Zb_{>0}^N.$ We define $X^n_{d,\mathbf p}$ as the closure of
$$\tilde X^n_{d,\mathbf p}=\left\{(f,\bfz_1,\ldots,\bfz_N)\in\mathcal P_d^n\times(\Cb^n)^N\,;\,\begin{array}{c}\bfz_i\text{ is a non-parabolic periodic point}\\ \text{of exact period}\ p_i\text{ of }f\text{ and no two }\\\bfz_j\text{ are in the same orbit of }f\end{array}\right\}.$$

The natural projection $\pi\colon X^n_{d,\mathbf p}\to\mathcal P_d^n$ defined a ramified cover and we are interested in the action of the monodromy on the fibers of this cover. More precisely, let $p_{\max}:=\max_{1\leq i\leq N}p_i$ and let 
$\mathcal P_{d}^n(p_{\max})\subset \mathcal P_d^n$ be the Zariski open subset of $\mathcal P_d^n$ that consists of all maps whose cycles of period less than or equal to $p_{\max}$ don't have eigenvalues that are equal to $1$. The analytic continuation of a $p$-periodic point, with $p\leq p_{\max}$, along a path in $\mathcal P_d^{n}(p_{\max})$ is well-defined. In particular, starting at the base point $F_0\in\mathcal P_d^{n}(p_{\max})$ defined by $F_0(z_1,\ldots,z_n)=(z_1^d,\ldots,z_n^d),$ the fundamental group of $\mathcal P_d^{n}(p_{\max})$ acts on the set of periodic points of period bounded by $p_{\max}$ of $F_0$ by permutations. An obvious constraint on this action is that it has to commute with the dynamics (after a permutation, the periods of periodic points cannot change, and the relative positions of periodic points in a cycle remain the same). Theorem~\ref{main_theorem_2} says that this is the only contraint. Actually, this action naturally extends to an action on the fiber $\pi^{-1}(F_0),$ that we call \emph{the action by monodromy} of $X^n_{d,\mathbf p}\to\mathcal P_d^n$ on its fibers, and we have the following result.

\begin{theorem}\label{main_theorem_2}
	For all $n, N\geq1$, $d\geq2$ and $\mathbf p=(p_1,\ldots,p_N)\in\Zb_{>0}^N$, the action by monodromy of $X^n_{d,\mathbf p}\to\mathcal P_d^n$ on its fibers is transitive. In particular $X^n_{d,\mathbf p}$ is irreducible. Moreover, the analogue result holds on $\mathrm{End}_d(\Pb^n)$.
\end{theorem}
In particular, it implies that the cover $X^n_{d,\mathbf p}\to\mathcal P_d^n$ is Galois, i.e. that its covering automorphisms act transitively on the fibers. 

So far, Theorem~\ref{main_theorem_2} has been known only for the case $n=1$: the case $n=1$ and $d=2$ has been proven in~\cite{bousch}, and the more general case $n=1$, $d\ge2$ was shown in~\cite{schleicher-galois} even in the more restrictive setting of unicritical polynomials instead of all degree $d$ polynomials $\mathcal P_d^1$ (see also~\cite{morton}). Observe that the case $n\geq2$ and $N=1$ over $\mathrm{End}_d(\Pb^n)$ has been solved by Fakhruddin in \cite{fakhruddin}. His strategy relies on the fact that in $\mathrm{End}_d(\Pb^n)$, there are more ways to deform the power map $F_0$  while preserving an invariant fibration. We were not able to adapt his proof neither to the case $N\geq2,$ nor to the polynomial case, nor for eigendirections (see below).

Using families that admit an invariant fibration, the one-dimensional result easily yields some permutations. The main difficulty in proving Theorem~\ref{main_theorem_2} lies in overcoming the restrictions imposed by the fibration on the possible permutations, which we address using dynamical arguments.

Finally, it is important to mention that our proof of Theorem~\ref{main_theorem_1} relies not only on the fact that all permutations of periodic points commuting with the dynamics can be implemented, but also on the possibility to interchange eigendirections of the Jacobian matrices at periodic points via a monodromy in $\mathcal P_d^n$. The latter is guaranteed by the irreducibility theorem that we state below. This theorem extends the result of Theorem~\ref{main_theorem_2}.

For $n\ge 2$, the tangent line bundle to $\bbC^n$ is diffeomorphic to $\bbC^n\times\Pb^{n-1}$.
Given $n\geq2,$ $d\geq2,$ $N\geq1$ and $\mathbf p=(p_1,\ldots,p_N)\in\Zb_{>0}^N,$ we define $Z^n_{d,\mathbf p}$ as the closure of
$$
\tilde Z^n_{d,\mathbf p}=\left\{(f,((\bfz_i,\bfv_i))_{1\leq i\leq N})\in\mathcal P_d^n\times(\Cb^n\times\Pb^{n-1})^N\,;\,\begin{array}{c}\bfz_i\text{ is a non-parabolic periodic point of}\\ \text{exact period}\
p_i\text{ of }f, \text{ no two } \bfz_j \text{ are in}\\\text{the same orbit of }f\text{ and } \bfv_i \text{ is a simple}\\\text{eigendirection of }D_{\bfz_i}f^{n_i}\end{array}\right\}.
$$

We prove the following:
\begin{theorem}\label{main_theorem_3}
	For all $n, d\geq2$, $N\geq1$ and $\mathbf p=(p_1,\ldots,p_N)\in\Zb_{>0}^N$, the action by monodromy of $Z^n_{d,\mathbf p}\to\mathcal P_d^n$ on its fibers is transitive. In particular $Z^n_{d,\mathbf p}$ is irreducible. Again,  the analogue result holds on $\mathrm{End}_d(\Pb^n)$.
\end{theorem}

Observe that some of the fibers of the projection $Z^n_{d,\mathbf p}\to\mathcal P_d^n$ have positive dimensions. This is in particular the case for the fiber above $F_0(z_1,\ldots,z_n)=(z_1^d,\ldots,z_n^d),$ which is also a source of minor technical difficulties arising in the proof of Theorem~\ref{main_theorem_1}. 
\subsection{Strategy of the proof of Theorem~\ref{main_theorem_1} and structure of the paper}\label{sec_strategy}

Any tuple of $n N_{d,n}$ eigenvalue functions corresponding to $n N_{d,n}$ distinct periodic orbits, can be viewed as a tuple of algebraic functions on an irreducible component of an appropriate algebraic variety $Z_{d,\mathbf p}^n$. (The vector $\mathbf p$ here is the vector of periods of the selected periodic orbits.) Thus, in order to prove Theorem~\ref{main_theorem_1}, it is sufficient for any selection of the eigenvalue functions to find a point on the corresponding irreducible component of $Z_{d,\mathbf p}^n$, at which these eigenvalues are \textit{locally independent} (i.e., the Jacobian matrix of the partial derivatives of the eigenvalue functions has full rank).

On the other hand, Theorem~\ref{main_theorem_3} which we prove in Section~\ref{sec_irreducibility} together with Theorem~\ref{main_theorem_2}, implies that the variety $Z_{d,\mathbf p}^n$ is irreducible. Hence, existence of a point in $Z_{d,\mathbf p}^n$, where the selected eigenvalue functions are locally independent, is implied by the following proposition.

\begin{proposition}\label{lambda_loc_indep_prop}
	Given any finite sequence of integers $p_1,\ldots,p_{nN_{d,n}} \ge 4$, there exists a map $F\in\mathcal P_d^n$ and a finite sequence of eigenvalue functions $\lambda_1,\ldots,\lambda_{nN_{d,n}}$ corresponding to distinct periodic orbits of the respective periods $p_1,\ldots,p_{nN_{d,n}}$, such that the eigenvalue functions $\lambda_1,\ldots,\lambda_{nN_{d,n}}$ are locally independent at $F$.
\end{proposition}

We give a proof of Proposition~\ref{lambda_loc_indep_prop} in Section~\ref{sec_main_proofs}. The proof is based on the local considerations near the power map $F_0(z_1,\ldots,z_n)=(z_1^d,\ldots,z_n^d)$. Note that the eigenvalue functions are not well defined in a neighborhood of the power map $F_0$, so the latter cannot be selected as the map $F$ in Proposition~\ref{lambda_loc_indep_prop}. We solve this problem by considering slightly modified functions that are well defined in a neighborhood of $F_0$ and agree well with the eigenvalue functions. The corresponding statement is proven in Section~\ref{sec_diag_entries}. In Section~\ref{sec_part_der} we compute the partial derivatives of the modified eigenvalue functions at the map $F_0$, and in Section~\ref{sec_independence} we show that there exists a selection of the modified eigenvalue functions with the prescribed periods, for which the Jacobian matrix is non-degenerate. In Section~\ref{sec_main_proofs}, we complete the proof of Proposition~\ref{lambda_loc_indep_prop} (as well as Theorem~\ref{main_theorem_1}) by passing back to the actual eigenvalue functions and selecting an appropriate map $F$ sufficiently close to $F_0$. Finally, in Section~\ref{sec_coro} we explain how Corollary \ref{cor-mu-interior} and Corollary \ref{cor-unif} can be deduced from \cite{gauthier2023sparsity} using Theorem \ref{main_theorem_1}.

\bigskip

\noindent\textbf{Acknowledgement.} The authors would like to thank Valentin Huguin for pointing out the difficulty in defining the moduli space of regular polynomial endomorphisms of $\mathbb C^n$ when $n\geq2.$

\section{Irreducibility of $\mathcal P_d^n$ with marked periodic points}\label{sec_irreducibility}

The main goal of this section is to give a proof of Theorems~\ref{main_theorem_2} and~\ref{main_theorem_3}, in both $\mathcal P_d^n$ and $\mathrm{End}_d(\Pb^n),$ focusing first on the former.

Recall that for a positive integer $p_{\max}>0$, the set 
$\mathcal P_{d}^n(p_{\max})\subset \mathcal P_d^n$ is the set of all maps whose cycles of period less than or equal to $p_{\max}$ don't have eigenvalues that are equal $1$. 
As we have already seen, the analytic continuation of a $p$-periodic point, with $p\leq p_{\max}$, along a path in $\mathcal P_d^{n}(p_{\max})$ is well-defined. 

The proof of Theorem~\ref{main_theorem_2} relies on its one-dimensional version that is proven in a stronger form in~\cite{bousch} and~\cite{schleicher-galois}. We state this result below:

\begin{theorem}\label{th-dim1}
	Let $d\geq2$ and $p_{\max}\ge 1$ be two integers. Let $U_d(p_{\max})$ be the set of $c\in\Cb$ such that the polynomial $f_c(z):=z^d+c$ does not have a parabolic cycle of period $\leq p_{\max}$. Let $P_{p_{\max}}$ be the set of periodic points of $f_0$ of period at most $p_{\max}.$ Then, the analytic continuation along loops in $U_d(p_{\max})$ starting at $0$ induce all permutations of $P_{p_{\max}}$ that commute with the dynamics. In particular, the same holds for loops in a larger space $\mathcal P_d^{1}(p_{\max})$ instead of $U_d(p_{\max})$.
\end{theorem}

Recall that the map $F_0\in\mathcal P_d^n$ is defined by 
$$
F_0\colon (z_1,\ldots,z_n)\mapsto (z_1^d,\ldots,z_n^d).
$$
In particular, it belongs to the direct product $(\mathcal P_d^1)^n$. 
Observe that many permutations of periodic points cannot be obtained if we only consider deformations in the space of product maps $(\mathcal P_d^1)^n$. For example, if $x_0$ and $y_0$ are respectively fixed and a $2$-periodic point for $z\mapsto z^d$ then $(x_0,y_0)\in\Cb^2$ is a $2$-periodic point for $F_0$ when $n=2.$ Its analytic continuation along a loop in $\mathcal P_d^{2}(p_{\max})$ of the form
$$F_t(x,y)=(f_t(x),g_t(y))$$
will be $(x_t,y_t)$ with $x_t$ fixed for $f_t$ and $y_t$ $2$-periodic for $g_t.$ Hence, $(x_0,y_0)$ cannot be exchanged in this way with $(y_0,y_0)$ which also has period $2.$ Moreover, if $(x_0,y_0)$ is exchanged with $(x_0',y_0')$ then all periodic points of the form $(x_0,y)$ are changed to $(x_0',y')$ for some $y'.$

To highlight this difficulty, we say that a periodic point $\mathbf z=(z_1,\ldots,z_n)$ of a product map $(x_1,\ldots,x_n)\mapsto(f_1(x_1),\ldots,f_n(x_n))$ has \emph{type} 
$$
\mathbf p=(p_1,\ldots,p_n)
$$
if each $z_i$ has period $p_i$ for $f_i.$

\subsection{Strategy of the proof of Theorem~\ref{main_theorem_2}}
We first prove the theorem in the case $n=2$ which captures the main difficulties that arise in the multi-dimensional case and are not present when $n=1$. In the proof we consider loops in $\mathcal P_d^n$ that start and end at $F_0$. The proof splits into two main steps: in Proposition~\ref{prop-same-type} we show that if two periodic points $\bfz=(x,y)$ and $\bfz'=(x',y')$ of type $(1,p)$ with either $x=x'$ or $y=y'$ belong to distinct cycles, then they can be interchanged by an appropriate loop so that all other cycles of periods $\le p_{\max}$ remain unchanged. On the other hand, if the above two points $\bfz$ and $\bfz'$ belong to the same cycle, then this cycle can be cyclically permuted by a loop so that $\bfz$ is sent to $\bfz'$ and again, all other cycles of periods $\le p_{\max}$ remain unchanged.

In the second step (Proposition~\ref{prop-change-type}) we prove that any two periodic points of the map $F_0$ with the same periods $p$ can be turned into periodic points $\bfz$ and $\bfz'$ of the kind described above, by following an appropriate loop in $\mathcal P_d^n$. Note that this loop is allowed to act nontrivially on other periodic orbits of period $\le p_{\max}$. Finally, by conjugating the loop from Proposition~\ref{prop-same-type} by the loops from Proposition~\ref{prop-facile}, we generate a sufficiently large collection of permutations of periodic points so that this collection generates the full group of permutations that commute with the dynamics.

The proof of Theorem~\ref{main_theorem_2} in $\mathcal P_d^n$ for $n\ge 3$ relies on the case $n=2$ and is given right after the proof for the case $n=2$. The extension to $\mathrm{End}_d(\Pb^k)$ is explained in Section \ref{sec-irre-Pn}

We would like to point out that one of the main difficulties in the proof is to avoid ``undesired'' permutations of periodic points (as in Proposition~\ref{prop-same-type}). This is done with the help of Lemma~\ref{le-hyp}, as also illustrated in Proposition~\ref{prop-facile} below. Furthermore, Lemma~\ref{le-hyp} also allows us to avoid undesired permutations of the eigendirections in the proof of Theorem~\ref{main_theorem_3} in the end of this section.

\subsection{Hyperbolicity condition for compositions of polynomials}

A key ingredient in order to use Theorem \ref{th-dim1} in higher dimension is the following elementary lemma. It will allow us to show that some paths are contractible in the shift locus of one variable polynomials.
\begin{lemma}\label{le-hyp}
	Let $d\geq2$ be an integer. Let $C>1.$ Let $K$ be a compact subset of $\Cb^*$ and $L$ be a compact in the set of monic degree $d$ polynomials. There exists a constant $A>0$ such that if $b\in\Cb$ satisfies $|b|>A$ and if $\alpha_1,\ldots,\alpha_N\in K$, $P_1,\ldots,P_N\in L$, for any $N\ge 1$, then the map $f:=f_N\circ\cdots\circ f_1,$ where $f_i(z):=P_i(z)-b\alpha_i$, is hyperbolic with an expanding constant larger than $C.$
\end{lemma}
\begin{proof}
	The proof simply consists to observe that under these assumptions, there exists large radii $R_0,R_1,\ldots,R_N$ such that $R_0=R_N$ and $f_i^{-1}(D(0,R_i))$ has $d$ relatively small connected components contained in $D(0,R_{i-1})$, for $i\in\{1,\ldots,N\}$. In this way, $f$ has a Cantor Julia set and is hyperbolic with a large expanding constant. And the construction of $R_0,\ldots,R_N$ simply comes from the fact that each $P_i$ looks like $z\mapsto z^d$ near $\infty.$
	
	In order to find these radii, let $d\geq2$ and let $\epsilon>0$ be small enough (depending on $d$) such that the set $\{z\in\Cb\,;\, |z^d-1|\leq2\epsilon\}$ is contained in the union of $d$ disjoint discs centered at the $d$-th roots of unity and of radius $\epsilon',$ for some $0<\epsilon'<1$.
	
	Observe that if $c\in\Cb^*$ and $z\in\Cb$ then $|z^d-c|\leq2\epsilon|c|$ is equivalent to $|1-z^d/c|\leq2\epsilon$ which implies that there exist a $d$-root $\xi$ of $1$ and a $d$-root $c^{1/d}$ of $c$ with $|\xi-z/c^{1/d}|\leq\epsilon'$, i.e. $|\xi c^{1/d}-z|\leq\epsilon'|c|^{1/d}.$ Moreover, Rouch\'e's theorem gives that if $h$ is a holomorphic function with $|h(z)|<\epsilon$ on $\partial D(\xi,\epsilon')$ then the image of $D(\xi,\epsilon')$ by $z\mapsto z^d+h(z)$ contains $D(1,\epsilon).$ Hence, there exists $A_0>0$ such that if $|c|>A_0$ and $P\in L$ then the preimage of $D(0,\epsilon|c|)$ by the map $w\mapsto P(w)-c$ is contained in the union of the $d$ discs $D(c^{1/d}\xi, \epsilon'|c|^{1/d}).$ In particular, it is contained in $D(0,2|c|^{1/d})$. Here, we applied the above argument to $h(z)=\sum_{i=0}^{d-1}a_ic^{(i-d)/d}z^i$ where $P(w)=w^d+\sum_{i=0}^{d-1}a_iw^i.$

	Now, let $\alpha_1,\ldots,\alpha_N\in K,$ $P_1,\ldots,P_N\in L$ and $b$ be in $\Cb^*.$ Define, for $i\in\{1,\ldots,N\}$, $R_i:=\epsilon|b\alpha_i|$ and $f_i\colon z\mapsto P_i(z)-b\alpha_i.$ In order to have $f_i^{-1}(D(0,R_i))\subset D(0,R_{i-1})$ it is sufficient to have $|b\alpha_i|>A_0$ and $2|\epsilon b\alpha_i|^{1/d}<\epsilon|b\alpha_{i-1}|$ which holds if
	$$|b|>\frac{2^{d/(d-1)}}{\epsilon}\left(\frac{\max_{1\leq i\leq N}|\alpha_i|}{\min_{1\leq i\leq N}|\alpha_i|^d}\right)^{1/(d-1)}.$$
	Hence, if $m=\min\{|z|\,;\,z\in K\}$ and $M=\max\{|z|\,;\,z\in K\},$ we can take $A:=\frac{4}{\epsilon}\left(\frac{M}{\min K^d}\right)^{1/(d-1)}.$
	A final observation is that $f^{-1}_i(D(0,R_i))$ is contained in $\Cb\setminus D(0,(1-\epsilon')|b\alpha_i|^{1/d})$ where the derivative of $f_i$ is $P_i'$ which becomes arbitrarily large, uniformly on $P_i\in L,$ if the constant $A$ is large enough.
\end{proof}

\subsection{Main steps of the proof: construction of the loops}

In the following proposition, we work in the space of skew-products and use Lemma \ref{le-hyp} in order to construct a target permutation, while avoiding the undesired ones.
\begin{proposition}\label{prop-facile}
	Let $p_{\max}$ and $p$ be two integers such that $1\leq p\leq p_{\max}.$ Let $\mathbf z_1$ and $\mathbf z_1'$ be two periodic points of $F_0$ of type $(1,\ldots,1,p)$ and of the form $\mathbf z_1=(0,\ldots,0,x_n)$ and $\mathbf z_1'=(0,\ldots,0,x_n').$ 
	
	If $\bfz_1$ and $\bfz_1'$ belong to different cycles, then there exists a loop in $\mathcal P_d^{n}(p_{\max})$ which exchanges $ \mathbf z_1$ and $ \mathbf z_1'$ while leaving all the other cycles of period smaller than or equal to $p_{\max}$ unchanged.
	
	If $\bfz_1$ and $\bfz_1'$ belong to the same cycle, then there exists a loop in $\mathcal P_d^{n}(p_{\max})$ which cyclically permutes this cycle, sending $ \mathbf z_1$ to $ \mathbf z_1'$ while leaving all the other cycles of period smaller than or equal to $p_{\max}$ unchanged.
	
\end{proposition}
\begin{proof}
	Let $\{\mathbf z_2,\ldots, \mathbf z_N\}\subset\Cb^n$ be the set of periodic points of period $\leq p_{\max}$ which are not in the same cycle as $\mathbf z_1$ nor $\mathbf z_1'.$ For each $j\in\{2,\ldots,N\},$ write $\mathbf z_j=(z_{j,1},\ldots,z_{j,n}).$
	Let $\tilde p$ denote the least common multiple of the periods of $\mathbf z_1,\ldots,\mathbf z_N.$ By Theorem \ref{th-dim1}, there exists a loop in $\mathcal P_d^{1}(\tilde p)$ of the form $f_t(x)=x^d+c(t),$ $t\in[0,1]$ with $c(0)=c(1)=0$, which exchanges the $p$-periodic points $x_n$ and $x_n',$ if they belong to different cycles or sends $x_n$ to $x_n'$ if they belong to the same cycle. Furthermore, the loop can be chosen so that all the periodic points of period smaller or equal to $\tilde p$ that are not in the same cycle as $x_n$ or $x_n'$, stay unchanged. 
	
	Observe that the loop 
	$$
	[0,1]\ni t\mapsto F_{\gamma_0(t)}(u_1,\ldots,u_k)=(u_1^d,\ldots,u_{k-1}^d,f_t(u_k))
	$$
	exchanges $\bfz_1$ and $\bfz_1'$ but will also act non-trivially on some other periodic orbits of periods less than or equal to $p_{\max}$. In order to avoid this, we will use Lemma \ref{le-hyp} and work in the space of skew-products.
	
	Choose $R>0$ large enough so that $c([0,1])\subset D(0,R).$ Let $\pi\colon\Cb^n\to\Cb^{n-1}$ be the projection $\pi(u_1,\ldots,u_n)=(u_1,\ldots,u_{n-1})$. By construction, if
	$$J_0:=\{2\leq j\leq N\,;\, \pi(\mathbf z_j)=0\}$$
	then the loop $F_{\gamma_0(t)}$ acts trivially on $\mathbf z_j$ with $j\in J_0.$ Let $J_1:=\{2,\ldots,N\}\setminus J_0.$ As $\Cb$ is an infinite filed, there exists a linear form $h\colon\Cb^{n-1}\to\Cb$ such that $h(\pi(\mathbf z_j))\neq0$ if $j\in J_1.$ It also follows from the definition of the set $J_0$ that $h(\pi(\mathbf z_j))=0$, for all $j\in J_0$. By Lemma \ref{le-hyp} applied to $K:=\{h(\pi(\mathbf z_j))\, ;\, j\in J_1\}$ and $L:=\{z^d+a\,;\, |a|\leq R\}$, if the modulus of $b_1\in\Cb$ is large enough then for all $t\in[0,1]$ and $j\in J_1,$ we have that
	$$f_{t,j}:=f_{t,j,\tilde p-1}\circ\cdots\circ f_{t,j,0}\ \text{ where }\ f_{t,j,i}(z):=z^d+c(t)+b_1h(\pi\circ F_0^i(\mathbf z_j))$$
	is hyperbolic. Moreover, as the loop $t\mapsto c(t)$ is contractible in $D(0,R),$ the loop $t\mapsto f_{t,j}$ is contractible in the space of hyperbolic polynomials and hence, does not permute any periodic points of $f_{0,j} = f_{1,j}$, for $j\in J_1$.
	
	Now we define a loop
	$$
	F_{\gamma_1(t)}(u_1,\ldots,u_n):=(u_1^d,\ldots,u_{n-1}^d,u_n^d+c(t)+b_1h(u_1,\ldots,u_{n-1})).
	$$
	Note that $F_{\gamma_1(0)} = F_{\gamma_1(1)}\neq F_0$, but for each $j\in J_1$, and any $z\in\bbC$, we have 
	$$
	F_{\gamma_1(t)}^{\tilde p}(z_{j,1},\ldots,z_{j,n-1}, z)=(z_{j,1},\ldots,z_{j,n-1},f_{t,j}(z)).
	$$
	Hence, by the previous discussion, this loop acts trivially on the periodic points of $F_{\gamma_1(0)}$ of the form $(z_{j,1},\ldots,z_{j,n-1},z).$ Furthermore, if $\delta_1$ is a path in $\mathcal P_d^{n}(p_{\max})$ between $F_0$ and $F_{\gamma_1(0)}$ of the form $(u_1,\ldots,u_n)\mapsto(u_1^d,\ldots,u_{n-1}^d,u_n^d+tb_1h(u_1,\ldots,u_{n-1}))$, $t\in[0,1]$, then the analytic continuation of a periodic point $(u_1,\ldots,u_n)$ has the form $(u_1,\ldots,u_{n-1},z)$ for some $z\in\Cb$, and in particular, the points $\mathbf z_1,$ $\mathbf z_1'$ and $\mathbf z_j$ with $j\in J_0$ stay unchanged as we move along the path $\delta_1$. Thus, the concatenation of $\delta_1$ with $\gamma_1$ and $\delta_1^{-1}$ gives a loop in $\mathcal P_d^{n}(p_{\max})$ which exchanges $\mathbf z_1$ and $\mathbf z_1'$, leaving unchanged all $\mathbf z_j$ with $j\in J_0\cup J_1.$ This concludes the proof.
\end{proof}
\begin{remark}\label{rk-Pk}
	An important observation is that all the deformations in $\mathcal P_d^{n}(p_{\max})$ we used in the proof above act trivially on the hyperplane at infinity when seen as endomorphisms of $\Pb^n$. This will also hold for all the other results in $\mathcal P_d^{n}$ below. Given that, it will be easy to extend the irreducibility results to the space $\mathrm{End}_d(\Pb^n)$ instead of $\mathcal P_d^n$.
\end{remark}

In the next two propositions, we restrict ourselves to $n=2.$ First, we use Lemma \ref{le-hyp} to change the types of orbits.

\begin{proposition}\label{prop-change-type}
	Let $p\geq1$ and $p_{\max}\geq p.$ Let $\bfz_1$ and $\bfz_2$ be two $p$-periodic points of $F_0$ of periodic type $(p_1,q_1)$ and $(p_2,q_2)$ respectively. Then there exists a loop in $\mathcal P_d^{2}(p_{\max})$ such that the analytic continuations $\check \bfz_1=(\check x_1,\check y_1)$ and $\check \bfz_2=(\check x_2,\check y_2)$ of $\bfz_1$ and $\bfz_2$ respectively are both of type $(1,p),$ with either $\check x_1=\check x_2$ or $\check y_1=\check y_2.$
\end{proposition}

Note that in Proposition~\ref{prop-change-type} we do not necessarily require that the points $\bfz_1$ and $\bfz_2$ belong to distinct cycles.

\begin{proof}[Proof of Proposition~\ref{prop-change-type}]
	Let $\bfz_i$ be of type $(p_i,q_i)$ for $i\in\{1,2\}$ like in the statement. Let $\epsilon>0$ be an arbitrarily small constant which will be fixed later. Let $f_\epsilon\colon\bbC\to\bbC$ be a degree $d$ polynomial having a period $p_1$-cycle $(w_1,\ldots,w_{p_1})$ such that $|w_i|<\epsilon$ for $i\in\{1,\ldots,p_1-1\}$ and $|1-w_{p_1}|<\epsilon.$ One can take $f_\epsilon(w)=aw(w^{d-1}-1)$ with $a=C\epsilon^{-1}$ where $C>1$ depends on $d.$ For $c\in\Cb,$ define
	$$G_{c,\epsilon}(x,y)=(f_\epsilon(x),y^d+xc).$$
	
	Note that when $c=0$, the map $G_{0,\epsilon}$ belongs to the product space $(\mathcal P_n^1)^2\subset \mathcal P_n^2$, so, using Theorem~\ref{th-dim1}, we can find a path $\delta$ from $F_0$ to $G_{0,\epsilon}$ in the product space $(\mathcal P_n^1)^2$ such that the analytic continuation of $\bfz_1$ along $\delta$ is $(w_1,y_1),$ where $y_1$ is a periodic point of $y\mapsto y^d$ of period~$q_1$. Observe that $y_1$ is also an $r_1$-periodic point for the map $y\mapsto y^{d^{p_1}}$, where $r_1$ satisfies $p_1r_1=p$. Using the monodromy in the non-parabolic unicritical degree $d^{p_1}$ polynomials, there exists a loop $\gamma$ in that space such that the $r_1$-periodic point $y_1$ for $y\mapsto y^{d^{p_1}}$ is exchanged with a $p$-periodic point $\hat y_1$ of $y\mapsto y^d$ (which also has period $r_1$ for $y\mapsto y^{d^{p_1}}$). Since, when $\epsilon$ is small the second coordinate of $G^{p_1}_{c,\epsilon}(w_1,y)$ is arbitrarily close to $y^{d^{p_1}}+c,$ the analytic continuation of $(w_1,y_1)$ along the loop $t\mapsto G_{\gamma(t),\epsilon}$ is $(w_1,\hat y_1)$ when $\epsilon>0$ is small enough. In particular, the periodic point $(w_1,\hat y_1)$ has type $(p_1,p).$ Coming back to $F_0$ through $\delta^{-1}$, we obtain that the analytic continuation of $\bfz_1$ via this loop is some periodic point  $\hat{\bfz}_1:=(x_1,\hat y_1)$ of type $(p_1,p).$

	The same construction allows us to pass from $\hat \bfz_1$ to $\check \bfz_1:=(\check x_1,\hat y_1)$ of type $(1,p)$. More specifically, we exchange the roles of $x$ and $y$ in the previous argument, and we replace $p_1$ by $p.$ The only other difference is that when we use the monodromy starting at the map $x\mapsto x^{d^p},$ the point $x_1$ is fixed for this map and we exchange it with a fixed point of $x\mapsto x^d.$
	
	So far, we have not considered the other periodic points, especially $\bfz_2.$ Let $\hat{\bfz}_2=(\hat x_2,\hat y_2)$ denote the analytic continuation of $\bfz_2$ along the loop exchanging $\bfz_1$ with $\check \bfz_1.$ We now have to exchange $\hat \bfz_2$ with a type $(1,p)$ periodic point of the form $(\check x_2,\hat y_1)$ or $(\check x_1,\check y_2),$ leaving $\check \bfz_1$ unchanged. Observe that if $\hat x_2=\check x_1$ then $(p_2,q_2)=(1,p)$ and $\hat \bfz_2$ already has the desired form. Hence, we can assume that $\hat x_2\neq\check x_1.$ (The latter will in particular imply that the points $\hat{\bfz}_2$ and $\check\bfz_1$ belong to different cycles.) In the same spirit, if $p_2=1,$ we can assume that $\hat y_2\neq\hat y_1,$ otherwise, there is nothing to prove. On the other hand, the map $f_\epsilon$ we choose above has a fixed point at $0$ and a $p_2$-cycle $(w_1',\ldots,w_{p_2}')$ such that $|w_i'|<\epsilon$ for $i\in\{1,\ldots,p_2-1\}$ and $|1-w_{p_2}'|<\epsilon.$ Using once again Theorem \ref{th-dim1} and the monodromy in $\mathcal P_d^{1}(p_{\max})$, there is a path from $z\mapsto z^d$ to $f_\epsilon$ such that $\check x_1$ becomes $0$ and $\hat x_2$ becomes $w_1'.$ From that, using as above a loop of the form $G_{\eta(t),\epsilon},$ we can exchange $(w_1',\hat y_2)$ with $(w_1',\hat y_1).$ The latter point is of type $(p_2,p)$ with the second coordinate matching the one of $\check\bfz_1$. Note that all the periodic points of the form $(0,y)$, in particular $(0,\hat y_1)$, remained unchanged by the monodromy along the loop $G_{\eta(t),\epsilon}$.
	
	For the last step, since the second coordinate of the two periodic points are the same, we can exchange the first coordinate of the second point so that it becomes of type $(1,p),$ leaving the first periodic point unchanged.
\end{proof}

\begin{proposition}\label{prop-same-type}
	Let $p_{\max}$ and $p$ be two integers such that $1\leq p\leq p_{\max}.$ Let $\bfz_1=(x_1,y_1)$ and $\bfz_1'=(x_1',y_1')$ be two periodic points of $F_0$ of type $(1,p)$ with $x_1=x_1'$ or $y_1=y_1'.$ 

	If $\bfz_1$ and $\bfz_1'$ belong to different cycles, then there exists a loop in $\mathcal P_d^{2}(p_{\max})$ which exchanges $ \mathbf z_1$ and $ \mathbf z_1'$ while leaving all the other cycles of period smaller than or equal to $p_{\max}$ unchanged.

	If $\bfz_1$ and $\bfz_1'$ belong to the same cycle, then there exists a loop in $\mathcal P_d^{2}(p_{\max})$ which cyclically permutes this cycle, sending $ \mathbf z_1$ to $ \mathbf z_1'$ while leaving all the other cycles of period smaller than or equal to $p_{\max}$ unchanged.

\end{proposition}
\begin{proof}
	Observe first that if the points $\bfz_1$ and $\bfz_1'$ belong to the same cycle, then $x_1=x_1'$. In general, the case $x_1=x_1'$ follows from Proposition \ref{prop-facile}: using Theorem~\ref{th-dim1} we can assume that $x_1=x_1'=0$ and then apply Proposition~\ref{prop-facile} with $n=2.$ Thus, for the rest of the proof we assume that the points $\bfz_1$ and $\bfz_1'$ belong to distinct cycles, $y_1=y_1'$ and $p\geq2$ since when $p=1$ the situation is similar to the case $x_1=x_1'.$
	
	As in the proof of Proposition \ref{prop-facile}, let $\{\bfz_2,\ldots, \bfz_N\}\subset\Cb^2$ be the set of periodic points of period $\leq p_{\max}$ which are not in the same cycles as $\bfz_1$ and $\bfz_1'.$
	
	It is classical fact (see the ``tour de valse'' \cite{douady-sentenac} and the proof of \cite[Corollary 4.11]{dujardin-bif} for precise details) that there exists a continuous family of maps $(g_\epsilon)_{\epsilon\in [0,\epsilon_0]}$ contained in $\mathcal P_d^1$ and parameterized by a small interval $[0,\epsilon_0]$ such that for each $\epsilon\in [0,\epsilon_0],$ the point $0$ is a fixed point of $g_\epsilon$ and if $\epsilon\neq0$ then $g_\epsilon$ has a $p$-cycle $(w_1(\epsilon),\ldots,w_p(\epsilon))$ whose points are contained in the disc $D(0,p\epsilon)$ and depend continuously on $\epsilon.$ Moreover, we can assume that $g_\epsilon\in\mathcal P_d^{1}(p_{\max})$ if $\epsilon\neq0.$ Observe that a priori, this family cannot be extended to the one parametrized holomorphically by $D(0,\epsilon_0)$ with the same property since $\epsilon\mapsto w_1(\epsilon)$ could be multivalued.

	For $c,\alpha\in\Cb$ and $\epsilon\in[0,\epsilon_0],$ we defined
	\begin{equation}\label{eq-g}
	G_{c,\alpha,\epsilon}(x,y):=(x^d+\alpha y+c,g_\epsilon(y)).
	\end{equation}
	By Theorem \ref{th-dim1}, there exists a path in $\mathcal P_d^1$ form $x\mapsto x^d$ to $g_{\epsilon_0},$ sending $y_1$ to $w_1(\epsilon_0).$ Let us denote by $\gamma$ the corresponding path in the product space $(\mathcal P_d^1)^2\subset\mathcal P_d^2$ between $F_0$ and $G_{0,0,\epsilon_0}$ such that the first coordinate remains unchanged along the path. For $i\in\{1,\ldots,N\},$ let $\hat \bfz_i=(x_i,\hat y_i)$ and $\hat \bfz_1'=(x_1',\hat y_1)$ be the analytic continuations along the path $\gamma$ of the points $\bfz_i$ and $\bfz_1'$ respectively. There exist natural analytic continuations $\hat \bfz_i(\epsilon)=(x_i,\hat y_i(\epsilon))$ and $\hat \bfz_1'(\epsilon)=(x_1',\hat y_1(\epsilon))$  of $\hat \bfz_i$ and $\hat \bfz_1'$ respectively, for $\epsilon\in[0,\epsilon_0]$ with $\hat \bfz_i(\epsilon_0)=\hat \bfz_i$ and $\hat \bfz_1'(\epsilon_0)=\hat \bfz_1'$. Among these periodic points, we isolate those with one ``bad'' coordinate by setting $\{2,\ldots,N\}=I_0\sqcup(I_1\cup I_2)$ where $i\in I_1$ iff $x_i=x_1$ or $x_i=x_1'$ and $i\in I_2$ iff $\hat y_i(0)=0.$ Observe that $i\in I_1\cap I_2$ implies that $\hat \bfz_i(\epsilon_0)$ is a fixed point equal to $(x_1,0)$ or $(x_1',0)$ since the orbits of the original periodic points $\bfz_1,\bfz_1'$ are disjoint from the points $\bfz_2,\ldots,\bfz_N$.
	
	For $i\not\in I_2$, let $n_i$ be the period of the periodic point $\bfz_i$ and let $f_{c,\alpha,i}\colon\bbC\to\bbC$ be the family of maps, parameterized by $c$ and $\alpha$, such that
	$$
	G_{c,\alpha,0}^{n_i}(x,\hat y_i(0))=(f_{c,\alpha,i}(x),\hat y_i(0)),\qquad\text{for all }x\in\bbC.
	$$
	Using Theorem \ref{th-dim1} in the space of non-parabolic unicritical degree $d$ polynomials $f_c$, there is a loop $c\colon[0,1]\to\Cb$ which exchanges $x_1$ and $x_1',$ leaving invariant all $x_i$ with $i\notin I_1.$ On the other hand, there exists $r>0$ such that $|g_0^j(\hat y_i(0))|\geq r$ if $j\geq0$ and $i\notin I_2.$ Hence, using Lemma \ref{le-hyp} in the same way as in the proof of Proposition \ref{prop-facile}, there exists $\alpha>0$ such that 
	for all $i\notin I_2,$ the loop 
	$$
	[0,1]\ni t\mapsto f_{c(t),\alpha, i} 
	$$
	is a contractible loop in the space of hyperbolic polynomials.
	Thus, the loop $t\mapsto G_{c(t),\alpha,0}$ (conjugated by a path from $G_{0,0,0}$ to $G_{0,\alpha,0}$ in product maps) acts trivially on the periodic points $\hat \bfz_i(0)$ if $i\notin I_2.$ This remains true under small perturbations so there exists $\epsilon\in(0,\epsilon_0)$ such that $t\mapsto G_{c(t),\alpha,\epsilon}$ is a loop along which the periodic points corresponding to $\hat \bfz_1(\epsilon)$ and $\hat \bfz_1'(\epsilon)$ are exchanged and those corresponding to $\hat \bfz_i(\epsilon),$ $i\notin I_2,$ stay unchanged. 
	
	Finally, we observe that the remaining case $i\in I_1\cap I_2,$ corresponds to the two fixed points $(x_1,0)$ and $(x_1',0)$ of $G_{0,0,\epsilon}$, and these two fixed points are swapped by the constructed loop. However, using Proposition \ref{prop-facile} it is easy to cancel this undesired permutation (see the beginning of the present proof about the case $y_1=y_1'$ and $p=1$).
	\end{proof}

We can now prove Theorem \ref{main_theorem_2} when $n=2.$ 
\begin{theorem}\label{th-irr2}
	For all $N\geq1$, $d\geq2$ and $\mathbf p=(p_1,\ldots,p_N)\in\Zb_{>0}^N$, the action by monodromy of $X^2_{d,\mathbf p}\to\mathcal P_d^2$ on its fibers is transitive. In particular, $X^2_{d,\mathbf p}$ is irreducible.
\end{theorem}
\begin{proof}
	Let $p_{\max}=\max\{p_1,\ldots,p_N\}.$ Recall that for $n=2$, the map $F_0\colon\bbC^2\to\bbC^2$ is defined by $F_0\colon (x_1,x_2)\mapsto (x_1^d, x_2^d)$.	
	It is enough to prove that, if $(F_0,\bfz_1,\ldots,\bfz_N)\in X^2_{d,\mathbf p}$ and $\bfz'_i$ is another $p_i$-periodic point of $F_0$ then 
	either $\bfz_i$ and $\bfz_i'$ belong to the same cycle and there exists a loop in $\mathcal P_d^{2}(p_{\max})$ that cyclically permutes this cycle, sending $\bfz_i$ to $\bfz_i'$, or $\bfz_i$ and $\bfz_i'$ belong to distinct cycles and there exists a loop in $\mathcal P_d^{2}(p_{\max})$ which exchanges $\bfz_i$ and $\bfz_i'$. In both cases the loop should leave the points $\bfz_j$ with $j\neq i$ unchanged as long as $\bfz_j$ is not in the same orbit than $\bfz_i'$.

To prove this, consider the loop $\gamma$ given by Proposition \ref{prop-change-type} such the monodromy along $\gamma$ turns $\bfz_i$ (resp. $\bfz_i'$) into some points $\bfx_i$ (resp. $\bfx_i'$) that are both of type $(1,p).$ For $1\leq j\leq N$ with $j\neq i,$ let $\bfx_j$ be the analytic continuation of $\bfz_j$ along $\gamma.$ Let $\delta$ be the loop given by Proposition \ref{prop-same-type}. It exchanges $\bfx_i$ and $\bfx_i'$ if they belong to distinct cycles or otherwise, sends $\bfx_i$ to $\bfx_i'$ in both cases leaving other cycles unchanged. Then, the loop $\gamma\delta\gamma^{-1}$ provides the necessary permutation of the periodic points.
\end{proof}
Using Proposition \ref{prop-facile}, the result extends to all dimensions.

\begin{proof}[Proof of Theorem \ref{main_theorem_2} when $n\geq3$]
	Let $(F_0,\bfz_1,\ldots,\bfz_N)$ be in $X_{d,\mathbf p}^n.$ Let $\bfz_i$ be one of the points in this data and let $\bfz_i'$ be another $p_i$-periodic point of $F_0$ which is not in the same cycle than $\bfz_j$ for all $j\neq i.$ The goal here is to obtain a loop $\gamma$ in $\mathcal P_d^{n}(p_{\max})$ putting both points $\bfz_i$ and $\bfz_i'$ in the form required to apply Proposition \ref{prop-facile} and then use the path $\gamma\delta\gamma^{-1}$ where $\delta$ is given by Proposition \ref{prop-facile}. Observe that this strategy cannot work when $d=2$ and $p_i=2$, since under these conditions there is a only one $2$-cycle of the form $(0,\ldots,0,x_n).$ We will explain the case $p_i=2$ separately.

	Let us write $\bfz_i$ and $\bfz_i'$ in coordinates: $\bfz_i=(x_1,\ldots,x_n)$ and $\bfz_i'=(y_1,\ldots,y_n)$, and define $(q_1,\ldots,q_n)$ and $(r_1,\ldots,r_n)$ as the periodic types of, respectively, $\bfz_i$ and $\bfz_i'.$ If we assume that $p_i\neq2$ then $\max_{1\leq j\leq n} q_j\neq2$ and we can assume that this maximum is equal to $q_n.$ From that, $(x_1,x_n)$ is a periodic point in dimension $2$ of period $s,$ the least common multiple of $(q_1,q_n).$ In the same way, $(y_1,y_n)$ is a periodic point of period $t.$ By Theorem \ref{th-irr2}, there exists a loop in $\mathcal P_d^{2}(p_{\max})$ sending $(x_1,x_n)$ to $(0,x_n')$ and $(y_1,y_n)$ to $(0,y_n'),$ for some periodic points $x_n'$ and $y_n'$ of $z\mapsto z^d$ of periods $s$ and $t$ respectively. 	Lifting this loop to the space $\mathcal P_d^n(p_{\max})$, we obtain a loop that sends $\bfz_i=(x_1,\ldots,x_n)$ to $(0, x_2,\ldots,x_{n-1},x_n')$ and $\bfz_i'=(y_1,\ldots,y_n)$ to $(0, y_2,\ldots,y_{n-1},y_n')$. The proof is concluded by induction on the number of nonzero coordinates of the periodic points to obtain the points in the desired form for which Proposition \ref{prop-facile} can be applied.

	It remains to consider the case $p_i=2$. 
	First, observe that when $\bfz_i$ and $\bfz_i'$ are in the same cycle, we can follow the construction above. Thus, assume that $\bfz_i$ and $\bfz_i'$ are not in the same cycle. Then, there exist $1\leq a<b\leq n$ such that the points $(x_a,x_b)$ and $(y_a,y_b)$ are both of period $2$ but in different cycles of the two-dimensional map $(z_1,z_2)\mapsto (z_1^d,z_2^d)$. Without loss of generality we can assume that $(a,b)=(n-1,n).$ The result in dimension $2$ allows to assume that $(x_{n-1},x_n)$ and $(y_{n-1},y_n)$ are both of periodic type $(1,2)$ with $x_{n-1}=0$, $y_{n-1}=1$ and $x_n=y_n$. Hence, for all $j\in\{1,\ldots,n-2\}$ both $(x_j,x_n)$ and $(y_j,y_n)$ have period $2$ so, using again Theorem \ref{th-irr2}, we can assume that they also have periodic type $(1,2).$ Hence, for each $j\in\{1,\ldots,n-2\}$, the points $(x_j,x_{n-1})$ and $(y_j,y_{n-1})$ are distinct fixed points (since $x_{n-1}=0$ and $y_{n-1}=1$), so by the $2$-dimensional case we can assume that $x_j=y_j=0$ for all $j\in\{1,\ldots,n-2\}.$ Finally, $(x_{n-1},x_n)$ and $(y_{n-1},y_n)$ are in different $2$-cycles so we can assume that $x_n=y_n$ and $x_{n-1}=0,$ $y_{n-1}=1.$ Hence, it remains to exchange the periodic points of the form $\bfz_i=(0,\ldots,0,x_n)$ and $\bfz_i'=(0,\ldots,0,1,x_n)$ while leaving all other finitely many marked cycles unchanged. In order to do this, we cannot apply Proposition \ref{prop-facile} directly. However, we give an argument that combines ideas from Proposition \ref{prop-facile} and Proposition~\ref{prop-same-type}. Certain details are identical to the ones that appear in the proofs of the aforementioned propositions and therefore are described very briefly.
	
	Same as in the proof of Proposition~\ref{prop-same-type}, let $g_\epsilon\in\mathcal P_d^1(p_{\max})$ be a polynomial with a fixed point at zero and an almost parabolic $2$-cycle located $\epsilon$-close to zero. Let $g_0$ be the limiting map, for which $0$ is a parabolic fixed point with multiplier $-1$. Note that for any sufficiently small $\epsilon\ge 0$, the map $g_\epsilon$ has no other periodic points of period $\le p_{\max}$ in a fixed neighborhood of zero.
	
	Let
	$$
	[0,1]\ni t\mapsto c(t)\in\bbC\qquad\text{with } c(0)=c(1)=0
	$$
	be a loop such that the corresponding loop $[0,1]\ni t\mapsto w^d+c(t)$ in the space of one-dimensional polynomials exchanges the fixed points $0$ and $1$, leaving all other periodic points of period $\le p_{\max}$ unchanged.
	
	Consider a linear form $h(w_1,\ldots,w_{n-2}, w_n)$ that is not equal to zero for any input $(w_1,\ldots,w_{n-2}, w_n)$, such that at least one coordinate of $(w_1,\ldots,w_{n-2}, w_n)$ is not equal to zero, and there exists $w_{n-1}\in\bbC$, such that $(w_1,\ldots,w_{n-2}, w_{n-1}, w_n)$ is a periodic point for 
	$$
	(z_1,\ldots,z_n)\mapsto (z_1^d,\ldots,z_{n-1}^d, g_0(z_n))
	$$
	with period $\le 2p_{\max}$.
	
	Finally, we construct a family of loops $[0,1]\ni t\mapsto F_{t,\epsilon,\alpha}\in\mathcal P_d^n(p_{\max})$, parameterized by $\alpha$ and $\epsilon$, where
	$$
	F_{t,\epsilon,\alpha} (w_1,\ldots,w_n) = (w_1^d,\ldots,w_{n-2}^d,w_{n-1}^d+\alpha h(w_1,\ldots,w_{n-2}, w_n) +c(t), g_\epsilon(w_n)).
	$$
	A path $\gamma$ that connects $F_0$ with $F_{0,\epsilon,\alpha}$ can be chosen so that it takes the $2$-periodic points $(0,\ldots,0,x_n)$ and $(0,\ldots,1,x_n)$ to respective almost parabolic points $(0,\ldots,0,\hat x_n)$ and $(0,\ldots,1,\hat x_n)$. Finally, using Lemma~\ref{le-hyp} as in the proofs of Propositions~\ref{prop-facile} and~\ref{prop-same-type}, one can select a sufficiently small $\epsilon>0$ and a sufficiently large $\alpha>0$, such that the loop $t\mapsto F_{t,\epsilon,\alpha}$ exchanges the points $(0,\ldots,0,\hat x_n)$ and $(0,\ldots,1,\hat x_n)$ while leaving all other periodic points $(w_1,\ldots,w_n)$ of periods $\le p_{\max}$ unchanged, provided that at least one of the coordinates $w_1,\ldots,w_{n-2}$ is nonzero or $|w_n|>\epsilon$. Thus, the only ``undesired'' permutation, generated by this loop, will be the permutation of the fixed points $(0,\ldots,0,0)$ and $(0,\ldots,1,0)$. We can reverse this permutation by applying the first part of the proof (i.e., the case, when the period is not equal to $2$).
\end{proof}

\subsection{Extension of Theorem~\ref{main_theorem_2} to the space $\mathrm{End}_d(\Pb^n)$}\label{sec-irre-Pn}

Next, we extend in Theorem~\ref{th-pk} the result of Theorem~\ref{main_theorem_2} to the space $\mathrm{End}_d(\Pb^n)$ of degree $d$ endomorphims of $\Pb^n.$

Exactly as above, for any $n\geq2,$ $d\geq2,$ $N\geq1$ and $\mathbf p=(p_1,\ldots,p_N)\in\Zb_{>0}^N,$ we define $Y^n_{d,\mathbf p}$ as the closure of
$$
\tilde Y^n_{d,\mathbf p}=\left\{(f,\bfz_1,\ldots,\bfz_N)\in\mathrm{End}_d(\Pb^n)\times(\Pb^n)^N\,;\,\begin{array}{c}\bfz_i\text{ is a non-parabolic periodic}\\ \text{point of exact period}\
p_i\text{ of }f\text{ and no}\\ \text{two }\bfz_j\text{ are in the same orbit of }f\end{array}\right\}.
$$
\begin{theorem}\label{th-pk}
	For all $n, N\geq1$, $d\geq2$ and $\mathbf p=(p_1,\ldots,p_N)\in\Zb_{>0}^N$, the action by monodromy of $Y^n_{d,\mathbf p}\to\mathrm{End}_d(\Pb^n)$ on its fibers is transitive. In particular $Y^n_{d,\mathbf p}$ is irreducible.
\end{theorem}
\begin{proof}
	Again, let $(F_0,\bfz_1,\ldots,\bfz_N)$ be in $Y_{d,\mathbf p}^n,$ $p_{\max}=\max\{p_1,\ldots,p_N\}$ and let $\bfz_i'$ be another $p_i$-periodic point of $F_0,$ that is not in the cycle of $\bfz_j$ when $j\neq i.$ If both $\bfz_i$ and $\bfz_i'$ are in $\Cb^n,$ then we can complete the proof by using the loop in $\mathcal P_d^{n}(p_{\max})$ given by Theorem~\ref{main_theorem_2}. For the general case, since the loop we consider in the proof of Theorem~\ref{main_theorem_2} acts trivially on periodic points at infinity (see Remark \ref{rk-Pk}), it is enough to find a loop $\gamma$ in the space of endomorphisms of $\Pb^n$ such that the analytic continuations of $\bfz_i$, $\bfz_i'$ and all other marked points $\bfz_j$ are well defined along $\gamma$, and in particular, analytic continuations of $\bfz_i$ and $\bfz_i'$ are both in $\Cb^n$. Then the loop $\gamma\delta\gamma^{-1}$ with the appropriate $\delta$ given by Theorem \ref{main_theorem_2}, provides the required permutation of the periodic points. Actually, using the monodromy in $\mathcal P_d^{n}(p_{\max}),$ it is enough to consider the case when $\bfz_i\in\Cb^n$ with only non-zero coordinates and $\bfz_i'$ is a point at infinity.
	
	Thus, if the hyperplane at infinity corresponds to $\{[x_0:\ldots:x_n]\in\Pb^n\,;\, x_0=0\}$ then we can assume that $\bfz_i=[y_0:\cdots:y_n]$ and $\bfz_i'=[y_0':\cdots:y_n']$ with $|y_j|=1$ for all $0\leq j\leq n$ and $y_0'=0.$ As $\bfz_i'$ has at least one non-zero coordinate, we can assume without loss of generality that $y_1'\neq0.$ For each $t\in\Cb\setminus\{1/2\},$ define $\phi_t\in\mathrm{Aut}(\Pb^n)$ by
	$$\phi_t[x_0:\cdots:x_n]=[(1-t)x_0+tx_1:tx_0+(1-t)x_1:x_2:\cdots:x_n].$$
	From this, if $\tilde\gamma$ is a path in $\Cb\setminus\{1/2\}$ between $0$ and $1$ then, using the fact that $F_0$ commutes with $\phi_1,$ we obtain that
	$$F_t:=\phi_{\tilde\gamma(t)}\circ F_0\circ\phi_{\tilde\gamma(t)}^{-1}$$
	defined a loop $\gamma$ in the space  of non-parabolic endomorphisms of $\Pb^n.$ And, it is easy to see that the analytic continuations of $\bfz_i$ and $\bfz_i'$ end at $[y_1:y_0:y_2:\cdots:y_n]$ and $[y'_1:y'_0:y'_2:\cdots:y'_n]$ respectively, thus, both are in $\Cb^n.$
\end{proof}

\subsection{Proof of Theorem~\ref{main_theorem_3}}

In the remaining part of this section we extend the result of Theorem~\ref{main_theorem_2} to give a proof of Theorem~\ref{main_theorem_3} in both cases, $\mathcal P_d^n$ and $\mathrm{End}_d(\Pb^n).$ Due to Theorem~\ref{main_theorem_2}, it remains to obtain permutations of the eigendirections. The idea is to find a loop around the locus where a specific periodic cycle has a differential with a Jordan block. Here again, the main difficulty is to accomplish this while leaving the eigendirections of any predetermined finite set of periodic cycles unchanged.

Given a positive integer $p_{\max}\ge 1$, consider a Zariski open subset $\tilde{\mathcal P}_d^n(p_{\max}) \subset \mathcal P_{d}^n(p_{\max})$  of $\mathcal P_{d}^n(p_{\max})$ that consists of all the maps whose cycles of period less than or equal to $p_{\max}$ all have distinct eigenvalues that are not equal to $1$. The conditions on the eigenvalues ensure that any eigendirection of any periodic point of period $\le p_{\max}$ can be analytically continued along any path in $\tilde{\mathcal P}_d^n(p_{\max})$. Note that the map $F_0$ that was used as a base point for the loops in the proof of Theorem~\ref{main_theorem_2}, is not contained in $\tilde{\mathcal P}_d^n(p_{\max})$. The latter is also a source of minor technical difficulties in the proof of Theorem~\ref{main_theorem_1} (see Section~\ref{sec_strategy}).

\begin{remark}\label{rk-jordan}
	Observe that, except in Theorem \ref{th-pk}, we had only used paths in skew-product maps (with $n-1$ independent coordinates). This can no longer be the case here because these maps have $n-1$ privileged eigendirections that stay unchanged along any loop in such skew products. Moreover, the loops in Theorem \ref{th-pk} are conjugated to loops used in Theorem \ref{main_theorem_2}. Thus, so far, all the paths we considered, whenever they lie in $\tilde{\mathcal P}_d^n(p_{\max})$, act trivially on the eigendirections.
\end{remark}

We first address the two-dimensional case. 
In what follows, we will consider loops 
of the form $[0,1]\ni t\mapsto F_{t,\epsilon,\alpha}$, defined by 
$$G_{t,\epsilon,\alpha}\colon (x,y)\mapsto (f(x)+\epsilon e^{2i\pi t}y,g(y)+\alpha x),$$
where $\epsilon>0$ is small and $\alpha>0$ is large. We will see in Proposition \ref{prop-dir} that if $0$ is a fixed point of $f$ with $\theta:=f'(0)$ non-zero and if $y_0$ is a $p$-periodic point of $g$ such that $(g^p)'(y_0)=\theta^p$ then the loop above exchanges the two eigendirections associated to the particular $p$-periodic point that is close to the periodic point $(0,y_0)$ of $(x,y)\mapsto (f(x),g(y)).$ 
To guarantee that the loop does not induce ``undesired'' exchanges of either periodic points or their eigendirections, we first carefully choose the map $g$, and then the constants $\alpha$ and $\epsilon$.

For $c\in\Cb,$ we consider the unicritical polynomial $g_c\colon z\mapsto z^d+c.$
Let $p$ be a positive integer and let $c_1\in\Cb$ be such that $0$ is a $p$-periodic point of $g_{c_1}.$ For $r>0$ small enough and $c\in D(c_1,r),$ $g_c$ has a $p$-periodic point $y(c)$ close to $0=y(c_1)$, depending holomorphically on $c,$ such that $\lambda(c):=(g_c^{p})'(y(c))\neq0$ if $c\neq c_1.$ For $c\in D(c_1,r)$ and $0\leq j\leq p-1$, define also $a_j(c):=g_{c}'(g_{c}^j(y(c))).$ Observe that, since the unique critical point of $g_c$ is $0,$ for $1\leq j\leq p-1,$ $a_j(c_1):=g_{c_1}'(g_{c_1}^j(0))\neq0.$ In particular, if $R>0$ is fixed then for $c$ sufficiently close to $c_1$, no $\theta\in\Cb$ satisfying $\theta^p=\lambda(c)$ is a root of a polynomial of the form
\begin{equation}\label{eq-poly}
X^{p-1}+\sum_{j=1}^{p-2}b_jX^j+\prod_{j=1}^{p-1}a_j(c),
\end{equation}
where $|b_j|<R$ for $1\leq j\leq p-2.$ In what follows, we fix such $c\neq c_1$ for
$$R=\sup\{1+|a_j(c')|\,;\, 0\leq j\leq p-1,\ c'\in D(c_1,r)\}^p.$$
From that we define $g:=g_c,$ $\lambda:=\lambda(c),$ $y_0:=y(c)$ and we choose $f\in\mathcal P_d^{1}$ without parabolic cycles and with a fixed point at $0$ and such that $\theta:=f'(0)$ satisfies $\theta^p=\lambda.$
\begin{proposition}\label{prop-dir}
	Let $p,$ $f$ and $g$ be as above and let $p_{\max}\geq p$ be an integer. If $\alpha>0$ is sufficiently large and $\epsilon>0$ sufficiently small, we have the following property. The loop $\gamma$ defined by
	$$G_t(x,y):=(f(x)+\epsilon e^{2i\pi t}y,g(y)+\alpha x)$$
	for $t\in[0,1]$ is in $\tilde{\mathcal P}_d^{2}(p_{\max}).$ Moreover, if $\bfz\in\Cb^2$ is a periodic point of $G_0$ of period at most $p_{\max}$ and $v\in\Pb^1$ is an eigendirection associated to $\bfz$ then
	\begin{itemize}
		\item the analytic continuation of $\bfz$ along $\gamma$ is $\bfz,$
		\item the analytic continuation of $v$ along $\gamma$ is $v$ unless $\bfz$ is a point in the cycle coming from the natural continuation of the cycle of the periodic point $(0,y_0)$ for $(x,y)\mapsto (f(x),g(y)+\alpha x).$ In this case, the action of $\gamma$ exchanges $v$ with the second eigendirection.
	\end{itemize}
\end{proposition}
\begin{proof}
	By Lemma \ref{le-hyp}, there exists $\alpha>0$ such that if $\tilde x\neq0$ is a periodic point of $f$ of period $l\leq p_{\max}$ then
	$$h:=h_{l-1}\circ\cdots\circ h_{0}\ \text{ where }\ h_{i}(z):=g(z)+\alpha f^i(\tilde x)$$
	is hyperbolic with an expanding constant larger than $|(f^l)'(\tilde x)|+1.$ Hence, since $f$ has no parabolic cycles, an $m$-periodic point $(\tilde x,\tilde y)$ of $H_0:=(x,y)\mapsto(f(x),g(y)+\alpha x)$ with $m\leq p_{\max}$ and $\tilde x\neq0$ can be followed with its eigendirections in a small neighborhood of $H_0.$ (Here, we use that $H_0^m(\tilde x,y)=(\tilde x,h^r(y))$ where $m=rl.$) The same statement also holds for $m$-periodic points of the form $(0,\tilde y)$, where $\tilde y$ is not in the orbit of $y_0$, since $g$ is unicritical and thus has a unique non-repelling cycle (the one that contains $y_0$).
	Thus, the only cycle which can be affected by a loop near $H_0$ is the cycle through the point $(0,y_0).$ As this cycle is attracting, it can also be followed locally so the only thing to prove is that the loop $\gamma$ in the statement swaps the eigendirections of this cycle.
	
	Recall that $\theta:=f'(0)$ verifies $\theta^p=\lambda:=(g^p)'(y_0).$ For $0\leq j\leq p-1$, define $a_j:=g'(g^j(y_0))$ in such a way that $\lambda=\prod_{j=0}^{p-1}a_j$ and
	$$D_{H_0^j(0,y_0)}H_0=\begin{pmatrix}\theta&0\\\alpha&a_j\end{pmatrix}\ \text{ and thus }D_{(0,y_0)}(H_0^p)=\begin{pmatrix}\lambda&0\\\alpha P&\lambda\end{pmatrix},$$
	where $P=\theta^{p-1}+\sum_{j=1}^{p-2}b_j\theta^j+\prod_{j=1}^{p-1}a_j$ with $b_j:=\prod_{i=j+1}^{p-1}a_i.$ Observe that the choice of $c$ to define $g(z):=z^d+c$ ensures that $\theta$ is not a root of a polynomial of the form \eqref{eq-poly}. Thus, $P\neq0.$
	
	Now, for $\epsilon\in\Cb$ define $H_\epsilon\colon(x,y)\mapsto(f(x)+\epsilon y,g(y)+\alpha x)$ and denote by $\bfz(\epsilon)=(x_\epsilon,y_\epsilon)$ the analytic continuation of $(0,y_0)$ as a $p$-periodic point, when $\epsilon$ is small. By the previous computation, there exist a constant $Q\in\Cb$ and holomorphic maps $\phi_i,$ $1\leq i\leq 4,$ such that
	$$D_{(x_\epsilon,y_\epsilon)}(H_\epsilon^p)=\begin{pmatrix}\lambda+\epsilon\phi_1(\epsilon)&\epsilon Q+\epsilon^2\phi_2(\epsilon)\\\alpha P+\epsilon\phi_3(\epsilon)&\lambda+\epsilon\phi_4(\epsilon)\end{pmatrix}.$$
	A simple computation gives that $Q=\theta^{p-1}+\sum_{j=1}^{p-2}\tilde b_j\theta^j+\prod_{j=0}^{p-2}a_j$ with $\tilde b_j:=\prod_{i=0}^{p-j-2}a_i.$ Given that $\theta^p=\lambda=\prod_{i=0}^{p-1}a_i,$ we have that
	$$\frac{\theta Q}{a_0}=\theta^{p-1}+\sum_{j=1}^{p-2}\tilde b'_j\theta^j+\prod_{j=1}^{p-1}a_j\ \text{ with } \ \tilde b'_j=\prod_{i=1}^{p-j-1}a_i.$$
	Thus, the fact that $\theta$ is not a root of a polynomial of the form \eqref{eq-poly} ensures again that $Q\neq0.$
	Furthermore, the two eigenvalues of $D_{(x_\epsilon,y_\epsilon)}(H_\epsilon^p)$ are
	$$\frac{2\lambda+\epsilon(\phi_1(\epsilon)+\phi_4(\epsilon))\pm\sqrt{4\epsilon\alpha PQ+\epsilon^2\phi_5(\epsilon)}}{2},$$
	where $\phi_5(\epsilon)=(\phi_1(\epsilon)-\phi_4(\epsilon))^2+4(Q\phi_3(\epsilon)+\alpha P\phi_2(\epsilon)+\epsilon\phi_2(\epsilon)\phi_3(\epsilon)).$ Since $PQ\neq0,$ if $\epsilon$ describes a loop of index $1$ around $0$ in $\Cb^*,$ close enough to $0,$ then the same holds for $4\epsilon\alpha PQ+\epsilon^2\phi_5(\epsilon).$ Hence, such a loop exchanges the two branches of the square root, i.e. it exchanges the two eigenvalues of $D_{(x_\epsilon,y_\epsilon)}(H_\epsilon^p)$ and thus the two eigendirections.
\end{proof}

\vspace{2ex}

Finally, we combine Proposition \ref{prop-dir} with Theorem \ref{main_theorem_2} and Theorem \ref{th-pk} to give a proof of Theorem~\ref{main_theorem_3}.

\begin{theorem}\label{th-eigen}
	For all $n, d\geq2$, $N\geq1$ and $\mathbf p=(p_1,\ldots,p_N)\in\Zb_{>0}^N$, the action by monodromy of $Z^n_{d,\mathbf p}\to\mathcal P_d^n$ on its fibers is transitive. In particular $Z^n_{d,\mathbf p}$ is irreducible. Moreover, the same holds for the corresponding space over $\mathrm{End}_d(\Pb^n).$
\end{theorem}
\begin{proof}
	Define $p_{\max}:=\max\{p_i\,;\, 1\leq i\leq N\}.$ We first fix $1\leq i\leq N$ and take $f,$ $g$ and $y_0$ as in Proposition \ref{prop-dir} with $p=p_i.$ We then choose $f_1,\ldots,f_n\in\mathcal P_d^{1}$ for which $0$ is a fixed point and such that the spectra
	$$S(h):=\{\lambda^{q}\,\mid\, 1\le q\le p_{\max}, \, \lambda \text{ is the multiplier of a }p\text{-periodic point of }h \text{ with }\ p\leq p_{\max}\}$$
	are pairwise disjoint for $h\in\{f,f_1,\ldots,f_n\}$ and do not contain $1.$ This holds for a generic choice of $(f_1,\ldots,f_n).$ It implies that the product map $F\colon(x_1,\ldots,x_n)\mapsto(f_1(x_1),\ldots,f_n(x_n))$ is in $\tilde{\mathcal P}_d^{n}(p_{\max}).$
	
	Now, let $(F,((\bfz_j,\bfv_j))_{1\leq j\leq N})$ be a point in $Z^n_{d,\mathbf p}$ over $F$ and recall that $i$ has been fixed above. Since $F$ is a product map, the eigendirection $\bfv_i$ corresponds to an axis $\bfe_s:=[0:\cdots:0:1:0:\cdots:0]$ with a $1$ at the $s$-th coordinate. Let $1\leq r\leq n$ be another integer. By Theorem \ref{main_theorem_2}, to prove the statement it is enough to find a loop in $\tilde{\mathcal P}_d^{n}(p_{\max})$ at $F$ which exchanges $(\bfz_i,\bfv_i)$ with $(\bfz_i,\bfe_r),$ leaving unchanged all the pairs $(\bfz_j,\bfv_j)$ for $j\neq i.$ This is what Proposition \ref{prop-dir} provides when $n=2$ but we have to pay attention to the extra dimensions.
	
	To simplify the notations, assume that $s=1$ and $r=2.$ 
	Using the assumption on the spectra above and Lemma \ref{le-hyp}, if $\alpha>0$ is large enough then the $p$-periodic points of
	$$G\colon(x_1,\ldots,x_n)\mapsto(g(x_1)+\alpha x_2,f(x_2),f_3(x_3),\ldots,f_n(x_n))$$
	with $p\leq p_{\max}$ only have simple eigenvalues. Hence, by Theorem \ref{main_theorem_2} and Remark \ref{rk-jordan}, there is a path in $\tilde{\mathcal P}_d^{n}(p_{\max})$ from $F$ to $G$ such that the analytic continuation of $(\bfz_j,\bfv_j)$ is $(\bfz_j',\bfv_j)$ where $\bfz_i'=(y_0,0,\ldots,0).$ Recall that $y_0$ is a $p_i$-periodic point for $g$ as in Proposition \ref{prop-dir} and that $0$ is a fixed point for $f$ and all $f_i.$ In order to only permute the direction of $\bfz_i'$, we continue the deformation to
	$$H\colon(x_1,\ldots,x_n)\mapsto(g(x_1)+\alpha x_2+\sum_{3\leq q\leq n}\beta_qx_q^2,f(x_2),f_3(x_3),\ldots,f_n(x_n)),$$
	where $0<\beta_3<\cdots<\beta_n$ are large constants obtained by induction using Lemma \ref{le-hyp} in the following way. For each $3\leq q\leq n,$ let $R_q>1$ be such that the annulus $D(0,R_q)\setminus D(0,R_q^{-1})$ contains all the $p$-periodic points of $f_q,$ with $p\leq p_{\max},$ except $0.$ $R_2$ is defined in the same way using $f.$ We set $K_2:=\overline{D(0,R_2)}\setminus D(0,R_2^{-1}),$ $K_q:=\overline{D(0,R_q^2)}\setminus D(0,R_q^{-2})$ and we choose a constant $C>0$ which largely dominates the derivative of $f$ and of the $f_q$ on their Julia sets. Then, $\beta_3$ is given by Lemma \ref{le-hyp} using the constant $C$ and the compacts $K_2$ and $L_2:=\{g+l\,;\, l\in \overline{D(0,\alpha R_2)}\}.$ By induction, $\beta_q$ is defined in the same way with $K_{q-1}$ and $L_{q-1}:=\{g+l\,;\, l\in\overline{D(0,\alpha R_2+\sum_{m=3}^{q-1}\beta_mR_m^2)}\}.$ These choices ensure that if $\epsilon>0$ is small enough then the loop $[0,1]\ni t\mapsto H_t$, where 
	$$H_t\colon(x_1,\ldots,x_n)\mapsto(g(x_1)+\alpha x_2+\sum_{3\leq q\leq n}\beta_qx_q^2,f(x_2)+\epsilon e^{2i\pi t} x_1,f_3(x_3),\ldots,f_n(x_n)),$$
	acts trivially on $(\bfz'_j,\bfv_j)$ if the last $n-2$ coordinates of $\bfz'_j$ are not simultaneously equal to $0.$ Moreover, since the coordinates $x_q$ are squared, if the last $n-2$ coordinates of $\bfz'_j$ vanish then the action of the loop on such points $\bfz'_j$ is the same as the one described in Proposition \ref{prop-dir}, i.e., it only exchanges $(\bfz'_i,\bfe_1)$ with $(\bfz'_i,\bfe_2).$ This concludes the proof over $\mathcal P_d^n.$

	The case of $\mathrm{End}_d(\Pb^n)$ follows in the same way using the arguments from the proof of Theorem \ref{th-pk}.
\end{proof}

\section{The diagonal entries of the Jacobian and the eigenvalue maps}\label{sec_diag_entries}

We recall that for a positive integer $p_{\max}>0$, the set 
$\mathcal P_{d}^n(p_{\max})\subset \mathcal P_d^n$ is the Zariski open subset of $\mathcal P_d^n$ that consists of all maps whose cycles of period less than or equal to $p_{\max}$ don't have eigenvalues that are equal $1$. We also consider a Zariski open subset $\tilde{\mathcal P}_d^n(p_{\max}) \subset \mathcal P_{d}^n(p_{\max})$  of $\mathcal P_{d}^n(p_{\max})$ that consists of all the maps whose cycles of period less or equal than $p_{\max}$ all have distinct eigenvalues that are not equal to $1$. 
The assumptions on the eigenvalues ensure that each periodic point of period less or equal than $p_{\max}$ can be followed locally and analytically in $\mathcal P_{d}^n(p_{\max})$, and every eigenvalue of such a point can be followed locally and analytically in the smaller subset $\tilde{\mathcal P}_{d}^n(p_{\max})$. Analytic continuation of either a periodic point or its eigenvalue is then well defined over $\mathcal P_{d}^n(p_{\max})$ and $\tilde{\mathcal P}_{d}^n(p_{\max})$ respectively and results in a (multiple valued) algebraic function.

More specifically, for a map $G_0\in \mathcal P_d^n(p_{\max})$ and a periodic point $\bfw_0\in\bbC^n$ of $G_0$ with period $p\le p_{\max}$, there exists a neighborhood $\mathcal U\subset \mathcal P_d^n(p_{\max})$ of $G_0$, such that analytic continuation of the periodic point $\bfw_0$ is well defined in $\mathcal U$ and results in an analytic (single valued) map $\mathcal U\ni G\mapsto \bfw(G)$ with $\bfw(G_0) = \bfw_0$.

For each index $k=1,\ldots,n$, one can consider an analytic function
$$
\rho_{k,\bfw_0}\colon \mathcal U\to\bbC,
$$
defined as the $k$-th diagonal entry of the Jacobian matrix $DG^p$ evaluated at the point $\bfw(G)$. Furthermore, if a neighborhood $\tilde{\mathcal U}\subset \mathcal U$ is simply connected and contained in $\tilde{\mathcal P}_{d}^n(p_{\max})$, $G_0\in\tilde{\mathcal U}$, and $\lambda_0\in\bbC$ is an eigenvalue of the Jacobian matrix $DG_0^p$, then the analytic continuation of this eigenvalue is well defined and gives an analytic map  
$$
\lambda\colon \tilde{\mathcal U}\to\bbC
$$
defined (and single valued) in the neighborhood $\tilde{\mathcal U}$ and satisfying $\lambda(G_0)=\lambda_0$.

A natural choice of local coordinates in $\mathcal U$ can be described as follows. 
For each index $m=1,\ldots, n$, let $\bfe_m$ denote the $m$-th coordinate unit vector in $\bbC^n$. For a multi-index $I = (i_1,\ldots,i_n)$ and an index $m\in\{1,\ldots,n\}$, let $\bfz^I$ denote
$$
\bfz^I = z_1^{i_1}\ldots z_n^{i_n}\qquad \text{and define}
$$
$$
P_{m,I}\colon\bbC^n\to\bbC^n\qquad\text{by}\qquad P_{m,I}(\bfz) = \bfz^I\bfe_m.
$$
If $\mathcal I_d^n$ is the set of multi-indices $I$, for which $P_{m,I}\in\mathcal P_d^n$, then all polynomials 
$$
\{P_{m,I}\mid I\in\mathcal I_d^n, 1\le m\le n\}
$$
form a local basis at any point of $\mathcal U$. Thus, for each $I\in\mathcal I_d^n$, $1\le m\le n$ we can consider the directional derivative operator $\partial_{m,I}$ in the direction of the polynomial $P_{m,I}$. In particular, for each $G\in\mathcal U$, we define
\begin{equation}\label{partial_rho_eq}
\partial_{m,I} \rho_{k,\bfw_0}(G):= \left.\frac{d}{dt}\right|_{t=0}\rho_{k,\bfw_0}(G+t P_{m,I}),
\end{equation}
and for each $G\in\tilde{\mathcal U}$, we define
\begin{equation}\label{partial_lambda_eq}
\partial_{m,I} \lambda(G):= \left.\frac{d}{dt}\right|_{t=0}\lambda(G+t P_{m,I}).
\end{equation}

\begin{remark}
	Note that the derivatives~(\ref{partial_rho_eq}) and~(\ref{partial_lambda_eq}) are well defined even when the multi-index $I$ does not belong to $\mathcal I_d^n$ (i.e., $I$ can have components greater than $d$). It will be assumed in all subsequent statements involving the above derivatives that the multi-index $I$ does not necessarily belong to $\mathcal I_d^n$, unless stated otherwise.	
\end{remark}

Next, consider a restricted class of maps that consist of all polynomials of the form

$$
F_\bfc (z_1,\ldots,z_n) = (z_1^d+c_1, z_2^d+c_2,\ldots, z_n^d+c_n),
$$
indexed by the vectors $\bfc = (c_1,\ldots,c_n)\in\bbC^n$.
Let $\mathcal A_d^n\subset \mathcal P_d^n$ be the set of all such maps, i.e.
$$
\mathcal A_d^n = \{F_\bfc \mid \bfc\in\bbC^n\}.
$$
For each $p_{\max}\ge 1$ we also introduce the sets
$$
\mathcal A_{d}^n(p_{\max}) := \mathcal A_d^n\cap \mathcal P_{d}^n(p_{\max})\qquad\text{and}
$$
$$
\tilde{\mathcal A}_{d}^n(p_{\max}) := \mathcal A_d^n\cap \tilde{\mathcal P}_{d}^n(p_{\max}). 
$$

\begin{lemma}\label{coincide_lemma}
	(1) The sets $\mathcal A_{d}^n(p_{\max})$ and $\tilde{\mathcal A}_{d}^n(p_{\max})$ satisfy $\tilde{\mathcal A}_{d}^n(p_{\max}) \subset \mathcal A_{d}^n(p_{\max})$ and are Zariski open in $\mathcal A_d^n$, for any integer $p_{\max}\ge 1$.
	
	(2) Given a periodic point $\bfw_0$ of $G_0\in \tilde{\mathcal A}_{d}^n(p_{\max})$ and an eigenvalue map $\lambda$ as above, defined in a neighborhood $\tilde{\mathcal U}$ of $G_0$ in $\tilde{\mathcal P}_{d}^n(p_{\max})$, there exists an index $k\in\{1,\ldots,n\}$, such that the following holds:
	\begin{itemize}
		\item $\lambda \equiv \rho_{k,\bfw_0}$ on $\tilde{\mathcal U}\cap\tilde{\mathcal A}_d^n(p_{\max})$;
		\item for any $m=1,\ldots,n$ and any multi-index $I$, we have $\partial_{m,I}\lambda \equiv \partial_{m,I}\rho_{k,\bfw_0}$ on $\tilde{\mathcal U}\cap\tilde{\mathcal A}_d^n(p_{\max})$.		
	\end{itemize}
\end{lemma}

\begin{proof}
The Jacobian matrix of any $G\in\mathcal A_d^n$ at any periodic point of period $p$ is diagonal; the elements on the diagonal of $G^p$ (i.e., the eigenvalues) are the multipliers of the corresponding periodic points for the one-dimensional maps $z\mapsto z^d+c_k$. By changing the constants $c_1,\ldots,c_n$, one can change these multipliers independently from each other, hence, the sets $\mathcal A_{d}^n(p_{\max})$ and $\tilde{\mathcal A}_{d}^n(p_{\max})$ are complements of the union of finitely many codimension $1$ algebraic subsets. This proves (1).

Assume, $G\in \tilde{\mathcal U}\cap\tilde{\mathcal A}_d^n(p_{\max})$ and the marked periodic point $\bfw = \bfw(G)$ has period $p$. Given a polynomial $P_{m,I}$ define a local one-parameter family of maps $G_t:= G+tP_{m,I}$, parameterized by $t\in (\bbC,0)$. According to the Implicit Function Theorem, the periodic point $\bfw(G_t)$ and the Jacobian matrix $J_t$ of $G_t^p$ at $\bfw(G_t)$ are well defined for $t$ close to zero. Due to the skew product structure of the maps $G_t$, it is not difficult to check that $J_t$ is a matrix whose only nonzero elements are located either on the diagonal, or in the $m$-th row. The eigenvalues of such a matrix are the diagonal elements, which immediately implies the statement of part (2).
\end{proof}

\section{Partial derivatives of the maps $\rho_{k,\bfw_0}$}\label{sec_part_der}

Consider the map $F_0\in\mathcal P_d^n$ defined by
$$
F_0(z_1,\ldots,z_n) = (z_1^d,\ldots, z_n^d).
$$

For any $p_{\max}\ge 1$, the map $F_0$ belongs to the difference $\mathcal A_d^n(p_{\max})\setminus \tilde{\mathcal A}_{d}^n(p_{\max})$ since it does not have multiple periodic points, but has fixed points and cycles with equal eigenvalues. 

For any periodic point $\bfw_0$ of $F_0$ with period $p\le p_{\max}$, there exists a neighborhood $\mathcal U\subset \mathcal P_d^n(p_{\max})$, such that for any index $k=1,\ldots, n$, the map $\rho_{k,\bfw_0}$ is well defined and analytic in $\mathcal U$.

Given a multi-index $I=(i_1,\ldots,i_n)$ as above, let $I_k$ denote the multi-index, obtained from $I$ by changing $i_k$ to $0$. That is, 
\begin{equation}\label{I_k_def_equation}
I_k = (i_1,\ldots,i_{k-1},0,i_{k+1},\ldots,i_n).
\end{equation}

The main technical result of this section is the following lemma.

\begin{lemma}\label{computation_lemma}
	Assume that in the above notation, none of the coordinates of a periodic point $\bfw_0 = (w_1,\ldots,w_n)\in\bbC^n$ are equal to zero. Let $p$ be the period of $\bfw_0$. Then for any multi-index $I = (i_1,\ldots,i_n)$ and any $k, m\in\{1,\ldots,n\}$, the following holds:
	\begin{equation*}
		\displaystyle\partial_{m,I}\rho_{k,\bfw_0}(F_0) = \begin{cases}
			0 & \text{if }m\neq k \\
			(i_kd^{p-1}-d^p)\sum_{i=0}^{p-1} (\bfw_0^{I_k})^{d^i} w_k^{d^i(i_k-d)} & \text{if } m=k.
		\end{cases}		
	\end{equation*}

\end{lemma}

\begin{proof}
	
	If $m\neq k$ then the lemma is obvious since in this case $\rho_{k,\bfw_0}$ is constant along the family of maps $F_t = F_0+t P_{m,I}$. Thus, we focus only on the case when $m=k$.
	
	Let $(\bfw_0(t), \bfw_1(t),\dots)$ be a periodic orbit of the polynomial $F_{t}=F_0+tP_{k,I}$ with $\bfw_i(t)=\bfw_{i+p}(t)$ for all integers $i$ and $\bfw_0(0)=\bfw_0$. For each index $i\in\bbN$ and $t\in(\bbC,0)$, we write $\bfw_i(t) = (w_{i,1}(t),\ldots,w_{i,n}(t))\in\bbC^n$. To simplify the notation, we will abbreviate $\bfw_i = \bfw_i(0)$ and $w_{i,s}=w_{i,s}(0)$.

	First we will compute the derivatives 
	$\left.\frac{dw_{i,s}(t)}{dt}\right|_{t=0}$. We note that
	$$
	\left.\frac{dw_{i,s}(t)}{dt}\right|_{t=0} = 0\qquad\text{for all }s\neq k.
	$$

	For $s=k$, since $w_{i+1,k}(t)=(w_{i,k}(t))^d + t(\bfw_i(t))^{I}$, we have
	$$
	\left.\frac{d w_{i+1,k}(t)}{d t}\right|_{t=0}= dw_{i,k}^{d-1}  \frac{d w_{i,k}}{d t}(0)+ \bfw_i^I.
	$$
	Since $w_{i,k}'(t)=w_{i+p,k}'(t)$, from the previous equality it follows that
	$$
	\left.\frac{d w_{i,k}(t)}{d t}\right|_{t=0}=
	\frac{\sum_{s=0}^{p-1}\bfw_{i+s}^I\prod_{r=s+1}^{p-1} \left(dw_{i+r,k}^{d-1} \right) }
	{1-\prod_{s=0}^{p-1}   \left(dw_{i+s,k}^{d-1} \right) }.
	$$
	Now we recall that $w_{i+s,k}=w_{i,k}^{d^s}$ and $w_{i,k}^{d^p}=w_{i,k}$, since $w_{i,k}$ is a periodic point of $z\mapsto z^d$. Since $w_{0,k}\neq 0$, we know that $w_{i,k}\neq 0$ for all $i\in\bbN$, hence $w_{i,k}^{d^p-1}=1$. Using this we get
	$$
	\left.\frac{d w_{i,k}(t)}{d t}\right|_{t=0}=
	\frac{\sum_{s=0}^{p-1}(\bfw_i^I)^{d^s}\prod_{r=s+1}^{p-1}dw_{i,k}^{d^{r+1}-d^r}}{1-d^p}=
	\frac{\sum_{s=0}^{p-1} d^{p-s-1} w_{i,k}^{d^p-d^{s+1}}(\bfw_i^I)^{d^s}}{1-d^p} =
	$$
	$$
	\frac{w_{i,k}\sum_{s=0}^{p-1} d^{p-s-1} w_{i,k}^{-d^{s+1}}(\bfw_i^I)^{d^s}}{1-d^p}=
	\frac{w_{0,k}^{d^{i}}\sum_{s=0}^{p-1} d^{p-s-1} w_{0,k}^{-d^{s+i+1}}(\bfw_0^I)^{d^{s+i}}}{1-d^p}.
	$$
	Now we can compute $\partial_{k,I}\rho_{k,\bfw_0}(F_0)$. 
	Observe that, using the definition of $I_k$ as in~(\ref{I_k_def_equation}), if $f_{t,j}(z) = z^d+ tz^{i_k}(\bfw_0^{I_k})^{d^j}$, then
	$$
	\rho_{k,\bfw_0}(F_t)=f_{t,0}'(w_{0,k}(t))  f_{t,1}'(w_{1,k}(t))\dots f_{t,p-1}'(w_{p-1,k}(t)),
	$$
	and by the derivative of a product formula we have 
	$$
	\partial_{k,I}\rho_{k,\bfw_0}(F_0)= \rho_{k,\bfw_0}(F_0)\sum_{i=0}^{p-1}\frac{f_{0,i}''(w_{i,k})\left.\frac{d w_{i,k}(t)}{d t}\right|_{t=0}+i_kw_{i,k}^{i_k-1}(\bfw_0^{I_k})^{d^i}}{f_{0,i}'(w_{i,k})}=
	$$
	$$
	d^p\sum_{i=0}^{p-1}\frac{d(d-1)w_{i,k}^{d-2}\left.\frac{d w_{i,k}(t)}{d t}\right|_{t=0}+i_kw_{i,k}^{i_k-1}(\bfw_0^{I_k})^{d^i}}{dw_{i,k}^{d-1}}=
	$$
	$$
	d^p\sum_{i=0}^{p-1}\frac{d(d-1)w_{0,k}^{d^i(d-2)}\cdot\frac{w_{0,k}^{d^i}\sum_{s=0}^{p-1} d^{p-s-1} w_{0,k}^{d^{s+i}(i_k-d)}(\bfw_0^{I_k})^{d^{s+i}}}{1-d^p} +i_kw_{0,k}^{d^i(i_k-1)}(\bfw_0^{I_k})^{d^i}} {dw_{0,k}^{d^i(d-1)}}=
	$$
	$$
	d^{p-1}\left(\sum_{i=0}^{p-1}\frac{d(d-1)\sum_{s=0}^{p-1} d^{p-s-1} w_{0,k}^{d^{s+i}(i_k-d)}(\bfw_0^{I_k})^{d^{s+i}}}{1-d^p} + i_kw_{0,k}^{d^i(i_k-d)}(\bfw_0^{I_k})^{d^i}\right)=
	$$
	$$
	i_kd^{p-1}\sum_{i=0}^{p-1}w_{0,k}^{d^i(i_k-d)}(\bfw_0^{I_k})^{d^i} + \frac{d^p(d-1)}{1-d^p}\sum_{s=0}^{p-1} d^{p-s-1} \sum_{i=0}^{p-1}w_{0,k}^{d^{s+i}(i_k-d)}(\bfw_0^{I_k})^{d^{s+i}}.
	$$
	Since $\bfw_0$ is a periodic point of $F_0$ of period $p$, it follows that the sum $\sum_{i=0}^{p-1}w_{0,k}^{d^{s+i}(i_k-d)}(\bfw_0^{I_k})^{d^{s+i}}$ is independent of $s$, and
	$$
	\sum_{i=0}^{p-1}w_{0,k}^{d^{s+i}(i_k-d)}(\bfw_0^{I_k})^{d^{s+i}} = \sum_{i=0}^{p-1}w_{0,k}^{d^{i}(i_k-d)}(\bfw_0^{I_k})^{d^{i}}.
	$$
	Hence,

	$$
	\partial_{k,I}\rho_{k,\bfw_0}(F_0)=
	\left(i_kd^{p-1} + \frac{d^p(d-1)}{1-d^p}\sum_{s=0}^{p-1} d^{p-s-1} \right) \sum_{i=0}^{p-1}w_{0,k}^{d^{i}(i_k-d)}(\bfw_0^{I_k})^{d^{i}} = 
	$$
	$$
	(i_kd^{p-1}-d^p)\sum_{i=0}^{p-1}w_{0,k}^{d^{i}(i_k-d)}(\bfw_0^{I_k})^{d^{i}}.
	$$
	We finish the proof by observing that $w_{0,k} = w_k$.
	
\end{proof}

Recall that the dimension of the moduli space $\tilde{\mathcal P}_d^n$ is 
$$
nN_{d,n} = n \left[{{d+n}\choose n}-n-1\right].
$$

Let $\widehat{\mathcal P}_d^n\subset\mathcal P_d^n$ be the space of maps
\begin{equation}\label{eq-monic}
(z_1,\ldots,z_n)\mapsto (f_1(z_1,\ldots,z_n),\ldots, f_n(z_1,\ldots,z_n)),
\end{equation}
where for each $k=1,\ldots, n$, 
the polynomial $f_k$ has a vanishing constant term, and contains the monomial $z_k^d$ with a constant coefficient $1$ and all monomials $z_j^d$ with $j\neq k$ with a constant coefficient $0$.

Note that $\dim \widehat{\mathcal P}_d^n = nN_{d,n}$.
For the fixed $d$ and $n$, a multi-index $I$ is called \textit{admissible}, if the map $F_0+tP_{k,I}$ belongs to $\widehat{\mathcal P}_d^n$, for any $t\in\bbC$.
One can easily check that a multi-index $I=(i_1,\ldots,i_n)$ is admissible if and only if $I\neq (0,0,\ldots,0)$ and $i_j\neq d$ for any $j=1,\ldots, n$.

\begin{lemma}\label{degree_count_lemma}
	Let $n,k\ge 1$ be integers, such that $k\le n$, and let $I=(i_1,\ldots,i_n)$ be an admissible multi-index. Assume that $\bfw_0 = (w_1,\ldots,w_n)$ is any periodic point of period $p\ge 1$ for $F_0$, such that $w_j\neq 0$ for any $j=1,\ldots,n$. Then
	$$
	\partial_{k,I}\rho_{k,\bfw_0}(F_0) = w_k^{-d^{p-1}}Q_{k,I}(\bfw_0),
	$$
	where $Q_{k,I}(z_1,\ldots,z_n)$ is a polynomial with 
	$$
	\deg_{z_j} Q_{k,I} = i_jd^{p-1},\qquad\text{for }j\neq k,\quad\text{and}
	$$
	$$
	\deg_{z_k} Q_{k,I} = \begin{cases}
		(i_k+1)d^{p-1}-1, & \text{for } 0\le i_k\le d-2 \\
		d^{p-1}-1 & \text{for } i_k=d-1.
	\end{cases}
	$$
	Furthermore, if $p\ge 2$, then for any two distinct admissible multi-indices $I$ and $I'$, the polynomials $Q_{k,I}$ and $Q_{k,I'}$ do not contain monomials that are proportional to each other.	
\end{lemma}

\begin{proof}
	We will write $\bfz$ for a vector $(z_1,\ldots,z_n)\in\bbC^n$. 
	The existence of the polynomial $Q_{k,I}$ follows immediately from Lemma~\ref{computation_lemma} and the fact that $w_k^{d^p}=w_k$: if $i_k\neq d-1$, then
	$$
	Q_{k,I}(\bfz) = (i_kd^{p-1}-d^p)\left((\bfz^{I_k})^{d^{p-1}} z_k^{d^{p-1}(i_k+1)-1} + \sum_{i=0}^{p-2} (\bfz^{I_k})^{d^i} z_k^{d^i(i_k-d)+d^{p-1}}  \right),
	$$	
	and if $i_k=d-1$, then 
	$$
	Q_{k,I}(\bfz) = -d^{p-1}\sum_{i=0}^{p-1} (\bfz^{I_k})^{d^i} z_k^{d^{p-1}-d^i}.
	$$	
	In both cases the corresponding degrees of the polynomials can be computed.
	
	Finally, suppose that for two distinct admissible multi-indices $I$ and $I'$, the polynomials $Q_{k,I}$ and $Q_{k,I'}$ have proportional monomials. Then these monomials must be of the form
	$$
	c(\bfz^{I_k})^{d^i}z_k^j \qquad\text{and}\qquad c'(\bfz^{I'_k})^{d^{i'}}z_k^{j},
	$$
	for some constants $c, c'\in\bbC$ and some indices $i, i', j$.
	
	If $I_k$ is not the zero vector, then since all elements of $I_k$ are not greater than $d-1$, it follows that $I_k=I'_k$ and $i=i'$. If $i<p-1$, then $j=d^i(i_k-d)+d^{p-1}=d^i(i'_k-d)+d^{p-1}$, which implies that $i_k=i'_k$, so $I=I'$, which is a contradiction. Now in the case $i=p-1$ we conclude from the formulas for $Q_{k,I}$ that $j=0$ if $i_k=d-1$ or $j=d^{p-1}(i_k+1)-1$ otherwise. Similar formulas hold when index $j$ is expressed through $i_k'$. Since $p\ge 2$ and $i_k<d$, these formulas imply that $i_k=i'_k$, hence, $I=I'$, which is a contradiction.
	
	Finally, if $I_k$ is the zero vector, then so is $I'_k$. In this case, since $I$ and $I'$ are admissible multi-indices, both $i_k$ and $i'_k$ are strictly positive.  
	This implies that if $j\ge 2d^{p-1}-1$, then $j=d^{p-1}(i_k+1)-1 =d^{p-1}(i'_k+1)-1$, hence, $i_k=i'_k=d-1$, and $I=I'$, which is a contradiction. Otherwise, if $j < 2d^{p-1}-1$, then $j=d^i(i_k-d)+d^{p-1}=d^{i'}(i'_k-d)+d^{p-1}$, for some indices $i$ and $i'$. Since $1\le i_k,i'_k\le d-1$, it follows that $i=i'$ and $i_k=i'_k$, hence, $I=I'$, which is a contradiction.
\end{proof}

\begin{remark}
	It follows from the proof that the requirement $p\ge 2$ in the last part of the lemma is essential.	
\end{remark}

\section{Independence of the maps $\rho_{k,\bfw_0}$ for selected periodic orbits}\label{sec_independence}

The main result of this section is Proposition~\ref{rho_independence_prop}, stated below.

For a fixed pair of integers $d$ and $n$, let $\mathcal I$ denote the set of all admissible multi-indices. Note that $|\mathcal I|=N_{d,n}$. Let
$$
I\colon \{1,2,\ldots,N_{d,n}\}\to \mathcal I
$$
be any fixed bijection (enumeration of all admissible multi-indices).

\begin{proposition}\label{rho_independence_prop}
	Given any finite sequence of integers 
	$$
	p_{1,1},\ldots,p_{1,N_{d,n}},p_{2,1},\ldots,p_{2,N_{d,n}},\ldots,p_{n,1},\ldots,p_{n,N_{d,n}},
	$$
	such that $p_{k,j}\ge 4$, for all pairs of $k\in[1,n]$ and $j\in[1,N_{d,n}]$,
	there exists a corresponding finite sequence  
	\begin{equation*}
		\bfw_{1,1},\ldots,\bfw_{1,N_{d,n}},\bfw_{2,1},\ldots,\bfw_{2,N_{d,n}},\ldots,\bfw_{n,1},\ldots,\bfw_{n,N_{d,n}},
	\end{equation*}
	of periodic points of $F_0$, belonging to distinct periodic orbits, and such that
	
	(i) for each $k\in\{1,\ldots, n\}$ and $j\in\{1,\ldots, N_{d,n}\}$, the period of $\bfw_{k,j}$ is~$p_{k,j}$;
	
	(ii) for each $k\in \{1,\ldots,n\}$, the maps $\rho_{k,\bfw_{k,1}}, \rho_{k,\bfw_{k,2}},\ldots, \rho_{k,\bfw_{k,N_{d,n}}}$ are locally independent at $F_0$. More specifically, the Jacobian matrix 
	$$
	J_k =  
	\begin{bmatrix}
		\partial_{k,I(1)}\rho_{k,\bfw_{k,1}}(F_0) & \partial_{k,I(2)}\rho_{k,\bfw_{k,1}}(F_0) & \ldots &
		\partial_{k,I(N_{d,n})}\rho_{k,\bfw_{k,1}}(F_0) \\
		\partial_{k,I(1)}\rho_{k,\bfw_{k,2}}(F_0) & \partial_{k,I(2)}\rho_{k,\bfw_{k,2}}(F_0) & \ldots &
		\partial_{k,I(N_{d,n})}\rho_{k,\bfw_{k,2}}(F_0) \\
		\ldots & \ldots & \ldots & \ldots \\
		\partial_{k,I(1)}\rho_{k,\bfw_{k,N_{d,n}}}(F_0) & \partial_{k,I(2)}\rho_{k,\bfw_{k,N_{d,n}}}(F_0) & \ldots &
		\partial_{k,I(N_{d,n})}\rho_{k,\bfw_{k,N_{d,n}}}(F_0) 		
	\end{bmatrix}
	$$
	has a nonzero determinant.
\end{proposition}

The proof of Proposition~\ref{rho_independence_prop} will be based on two technical lemmas, stated and proved below. First, we start with a definition:

\begin{definition}
	For a positive integers $s\in\bbN$, we say that a polynomial $P(z_1,\ldots,z_n)$ is an \textit{$s$-polynomial} if $\deg_{z_j} P \le s$, for any $j= 1, 2,\ldots, n$.

\end{definition}

\begin{remark}
	Note that the polynomials $Q_{k,I}$ from Lemma~\ref{degree_count_lemma} are $(d^p-d^{p-1})$-polynomials.	
\end{remark}

\begin{lemma}\label{choice_lemma}
	Assume, $d$ and $p$ are two integers such that $d\ge 2$ and $p\ge 2$. Then for every non-identically zero $(d^p-d^{p-1})$-polynomial $P\colon\bbC^n\to\bbC$, there exist at least $(d^{p-1} - d^{[p/2]})(d^{p-1}-1)^{n-1}$ periodic points $\bfw=(w_1,\ldots, w_n)$ of period $p$ for $F_0\in\mathcal P_d^n$, such that $w_1w_2\ldots w_n\neq 0$ and $P(\bfw)\neq 0$.	
\end{lemma}
\begin{proof}

	Define the sets
	$$
	Per_p:= \{w\in\bbC\mid w \text{ is a periodic point of period }p \text{ for the map }z\mapsto z^d\}
	$$
	$$
	Fix_p:= \{w\in\bbC\mid w\neq 0 \text{ and } w^{d^p}=w \}.
	$$
	Note that since $p>1$, we have $0\not\in Per_p$. Since the map $z\mapsto z^d$ has no multiple periodic points, it follows that
	$$
	|Fix_p| = d^p-1\qquad\text{and}\qquad |Per_p|\ge d^p- d^{[p/2]},
	$$
	where the square brackets denote the integer part. (The second inequality follows immediately from the formula $d^p = \sum_{k|p} |Per_k|$, where the summation is taken over all $k\ge 1$ that divide $p$.) The set
	$$
	S:= Per_p\times (Fix_p)^{n-1}\subset\bbC^n
	$$
	consists of periodic points of period $p$ for the map $F_0$ and none of the coordinates of the points from $S$ are equal to zero. We will show that the set $S$ contains the required number of points at which the polynomial $P$ does not vanish.
	
	Express $P$ as a polynomial of $n-1$ variables $z_2, z_3,\ldots, z_n$ with coefficients from the ring of polynomials of the remaining variable $z_1$. Since $P$ is not an identical zero, at least one of these coefficients is not an identical zero as well. Let us call it $q_1(z_1)$. 
	Since $P$ is a $d^p-d^{p-1}$-polynomial, we have $\deg q_1\le d^p-d^{p-1}$. 
	Comparing it to the size of the set $Per_p$, we conclude that there exist at least $d^{p-1} - d^{[p/2]}$ elements of $Per_p$ at which the polynomial $q_1$ does not vanish. Let $w_1\in Per_p$ be any of such points. Then the polynomial $P_1(z_2,z_3,\ldots,z_n)=P(w_1,z_2,z_3,\ldots,z_n)$ is not an identical zero.
	
	Next, we express $P_1$ as a polynomial of $n-2$ variables $z_3,z_4,\ldots,z_n$ with coefficients from the ring of polynomials of the variable $z_2$. Again, since $P_1$ is not identically zero, there will be a coefficient $q_2(z_2)$ that is not an identical zero. Again, since $P$ is a $d^p-d^{p-1}$-polynomial, we have $\deg q_2\le d^p-d^{p-1}$, and since $|Fix_p|=d^p-1$, there exist at least $d^{p-1}-1$ elements of $Fix_p$ at which the polynomial $q_2$ does not vanish. If $w_2\in Fix_p$ is any of such points, then one can consider the polynomial $P_2(z_3,z_4,\ldots,z_n)=P(w_1,w_2,z_3,z_4,\ldots,z_n)$ which is not identically zero and proceed in a similar way for all remaining variables $z_3,\ldots,z_n$. As a result, it follows that there exist $(d^{p-1} - d^{[p/2]})(d^{p-1}-1)^{n-1}$ elements of the set $S$ at which the polynomial $P$ does not vanish.	
\end{proof}

\begin{lemma}\label{ineq_lemma}
	For any $p\ge 4$, $d\ge 2$ and $n\ge 2$, the following inequality holds:
	$$
	pnN_{d,n} < (d^{p-1} - d^{[p/2]})(d^{p-1}-1)^{n-1}
	$$
\end{lemma}
\begin{proof}
	Observe that
	\begin{equation}\label{ineq_1}
		pnN_{d,n} < pn{{d+n}\choose n} < pn(d+1)^{n-1}\frac{d+n}{n} = p(d+n)(d+1)^{n-1}.		
	\end{equation}

	By direct computation one can verify that 
	$$
	\frac{d}{dx} (d+x)^{1/(x-1)} <0,\qquad \text{when } d\ge 2\text{ and }x\ge 2.
	$$
	Thus, $(d+n)^{1/(n-1)} \le d+2$, and from~(\ref{ineq_1}) it follows that
	$$
	pnN_{d,n} < p [(d+2)(d+1)]^{n-1}.
	$$
	If $p\ge 5$ or $d\ge 3$, then $(d+2)(d+1)\le d^{p-1}-1$, so
	$$
	pnN_{d,n} < p (d^{p-1}-1)^{n-1}\le (d^{p-1} - d^{[p/2]})(d^{p-1}-1)^{n-1}
	$$
	as required.

	In the remaining case $p=4$, $d=2$, the original inequality transforms into $n N_{2,n} < 7^{n-1}$. When $n=2$, this inequality can be verified by a direct computation. For $n\ge 3$, we have
	$$
	n N_{2,n} <n{{2+n}\choose n}<n\cdot 3^{n-1}\frac{2+n}{n}=3^{n-1}(2+n) <7^{n-1},
	$$
	which completes the proof.
\end{proof}

\begin{proof}[Proof of Proposition~\ref{rho_independence_prop}]

	In the proof we assume that $n\ge 2$. The case $n=1$ is fully covered in~\cite{gorbovickis-poly}.
	
	For $k=m\in\{1,\ldots,n\}$ we assume that the periodic points $\bfw_{i,j}$ with $i<m$ are already selected so that properties (i) and (ii) are satisfied. Now we will choose the periodic points $\bfw_{m,1},\ldots,\bfw_{m,N_{d,n}}$ so that $\det J_m\neq 0$. The proof will go by finite induction on the second index (i.e., from $1$ to $N_{d,n}$). The points $\bfw_{m,1},\ldots,\bfw_{m,N_{d,n}}$ will a posteriori belong to distinct periodic orbits, since otherwise $\det J_m$ cannot be nonzero. An appropriate counting argument will also imply that these points can be chosen so that they do not belong to the periodic orbits, that were previously selected for the smaller values of $k$.
	
	Fix $k=m\in\{1,\ldots,n\}$. The base of the induction argument can be done as follows:
 
	observe that $\partial_{m,I(1)}\rho_{m,\bfw_{m,1}}(F_0)$ is a $(d^{p_{m,1}}-d^{p_{m,1}-1})$-polynomial when viewed as a polynomial of the coordinates of the $p_{m,1}$-periodic point $\bfw_{m,1}$. Hence, according to Lemma~\ref{choice_lemma} there are at least $(d^{p_{m,1}-1} - d^{[p_{m,1}/2]})(d^{p_{m,1}-1}-1)^{n-1}$ choices for $\bfw_{m,1}$, such that 
	\begin{equation}\label{base_eq}
	\partial_{m,I(1)}\rho_{m,\bfw_{m,1}}(F_0) \neq 0.	
	\end{equation}
	On the other hand, the total number of previously selected periodic points $\bfw_{i,j}$ of period $p_{m,1}$ together with the other points from their orbits does not exceed $(m-1)p_{m,1}N_{d,n}\le (n-1)p_{m,1}N_{d,n}$. Then Lemma~\ref{ineq_lemma} implies that one can select the point $\bfw_{m,1}$ so that~(\ref{base_eq}) still holds and $\bfw_{m,1}$ does not belong to the periodic orbits of $\bfw_{i,j}$ with $i<m$.
	
	Now we do the induction step.
	Assume that for $j\in\{1,\ldots,N_{d,n}-1\}$, the periodic points $\bfw_{m,1},\ldots,\bfw_{m,j}$ are selected so that the matrix
	$$
	J_{m,j} =  
	\begin{bmatrix}
	\partial_{m,I(1)}\rho_{m,\bfw_{m,1}}(F_0) &  \ldots &
	\partial_{m,I(j)}\rho_{m,\bfw_{m,1}}(F_0) \\
	\partial_{m,I(1)}\rho_{m,\bfw_{m,2}}(F_0) &  \ldots &
	\partial_{m,I(j)}\rho_{m,\bfw_{m,2}}(F_0) \\
	\ldots & \ldots & \ldots \\
	\partial_{m,I(1)}\rho_{m,\bfw_{m,j}}(F_0) &  \ldots &
	\partial_{m,I(j)}\rho_{m,\bfw_{m,j}}(F_0) 		
	\end{bmatrix}
	$$
	has a nonzero determinant. Then, together with the second part of Lemma~\ref{degree_count_lemma}, this implies that the co-factor expansion of $\det J_{m,j+1}$ along the $(j+1)$-st row is a non-identically zero polynomial of the coordinates of the point $\bfw_{m,j+1}$. Furthermore, it follows from Lemma~\ref{degree_count_lemma} that this is a $(d^{p_{m,j+1}}-d^{p_{m,j+1}-1})$-polynomial, so exactly the same argument as in the previous paragraph shows that one can select the periodic point $\bfw_{m,j+1}$ so that $\det J_{m,j+1}\neq 0$ and $\bfw_{m,j+1}$ does not belong to any of the previously selected periodic orbits of the points $\bfw_{i,s}$ with $i<m$. This completes the proof of the induction step, and since $J_m=J_{m, N_{d,n}}$, the proposition follows.	
\end{proof}

\begin{remark}\label{ineq_remark}
	For the proof of Proposition~\ref{rho_independence_prop}, it is sufficient to use the inequality 
	$$
	(p-1)nN_{d,n} < (d^{p-1} - d^{[p/2]})(d^{p-1}-1)^{n-1}
	$$
	which is weaker than the one from Lemma~\ref{ineq_lemma}. We will need the inequality from Lemma~\ref{ineq_lemma} for the proof of Proposition~\ref{rho_independence_prop2} below which is required for establishing Theorem~\ref{main_theorem_Pn}. 
\end{remark}

\section{Proof of Theorem~\ref{main_theorem_1} and Theorem~\ref{main_theorem_Pn}}\label{sec_main_proofs}

\subsection{Polynomial case}

\begin{proof}[Proof of Proposition~\ref{lambda_loc_indep_prop}]

	Given a finite sequence of periods $p_1,\ldots,p_{nN_{d,n}} \ge 4$, let 
	\begin{equation*}
	\bfw_{1,1},\ldots,\bfw_{1,N_{d,n}},\bfw_{2,1},\ldots,\bfw_{2,N_{d,n}},\ldots,\bfw_{n,1},\ldots,\bfw_{n,N_{d,n}},
	\end{equation*}
	be the corresponding finite sequence of periodic points provided by Proposition~\ref{rho_independence_prop}. Consider the corresponding functions
	\begin{equation}\label{rho_functions_eq}
	\rho_{1,\bfw_{1,1}},\ldots, \rho_{1,\bfw_{1,N_{d,n}}},
		\rho_{2,\bfw_{2,1}},\ldots, \rho_{2,\bfw_{2,N_{d,n}}},\ldots,
		\rho_{n,\bfw_{n,1}},\ldots, \rho_{n,\bfw_{n,N_{d,n}}}.
	\end{equation}
	It follows from Proposition~\ref{rho_independence_prop} that these functions are locally independent at $F_0$. Indeed, according to Lemma~\ref{computation_lemma}, the Jacobian matrix for these functions is block-diagonal with blocks $J_1,\ldots,J_n$ as in Proposition~\ref{rho_independence_prop}. Since according to the same proposition, each block has a non-zero determinant, the same also holds for the whole Jacobian matrix.

	Note that $F_0\in\mathcal A_d^n(p_{\max})\setminus \tilde{\mathcal A}_d^n(p_{\max})$, where $p_{\max} = \max_{1\le j\le nN_{d,n}}\{p_{j}\}$, so according to part (1) of Lemma~\ref{coincide_lemma}, there exist maps in $\tilde{\mathcal A}_d^n(p_{\max})$ arbitrarily close to $F_0$. At the same time, part (2) of Lemma~\ref{coincide_lemma} implies that the Jacobian matrix of the functions~(\ref{rho_functions_eq}) coincides on $\tilde{\mathcal A}_d^n(p_{\max})$ with the Jacobian matrix of some eigenvalue functions of the same periodic orbits. Now, by continuity of the partial derivatives, it follows that these eigenvalue functions will be locally independent at a map $F\in \tilde{\mathcal A}_d^n(p_{\max})$ sufficiently close to $F_0$.	
\end{proof}

\begin{proof}[Proof of Theorem~\ref{main_theorem_1}]
	On $\mathcal P_d^n$, the theorem follows directly from Theorem~\ref{main_theorem_3} and Proposition~\ref{lambda_loc_indep_prop} exactly as discussed in Section~\ref{sec_strategy}.
\end{proof}

\subsection{The case of endomorphisms of $\Pb^n$}

The proof of Theorem~\ref{main_theorem_Pn} follows the same approach as the proof of Theorem~\ref{main_theorem_1}. Below we outline the key differences in the proofs.

We use the standard homogeneous coordinates $[z_0\colon z_1\colon\ldots\colon z_n]$ in $\Pb^n$. The subset $\{[z_0\colon \ldots\colon z_n]\in\Pb^n \mid z_0\neq 0\}$ is naturally identified with $\bbC^n$ via the projection
$$
[z_0\colon z_1\colon\ldots\colon z_n] \mapsto (z_1/z_0,\, z_2/z_0,\,\ldots,\, z_n/z_0)\in\bbC^n
$$
to the affine chart.

To every multi-index $I=(i_1,\ldots,i_n)\in\bbZ_{\ge 0}^n$ with $|I|=\sum_{j=1}^ni_j \le d$, we can associate a unique multi-index $\tilde I = (i_0, i_1,\ldots,i_n)\in\bbZ_{\ge 0}^{n+1}$ with $|\tilde I|:=\sum_{j=0}^ni_j = d$. A multi-index $\tilde I$ respectively determines $I$. As before, $\bfz^I$ and $\tilde\bfz^{\tilde I}$ will denote
$$
\bfz^I = z_1^{i_1}\dots z_n^{i_n}\qquad\text{and}\qquad \tilde\bfz^{\tilde I} = z_0^{i_0}z_1^{i_1}\dots z_n^{i_n}.
$$
Similarly, we can consider the monomials
$$
P_{m,I}(\bfz) = \bfz^I\bfe_m,\qquad\text{for }m=1,\ldots,n,\qquad \text{and}
$$
$$
P_{m,\tilde I}(\tilde\bfz) = \tilde\bfz^{\tilde I}\bfe_m,\qquad\text{for }m=0,\ldots,n.
$$

Let $\tilde{\mathcal I}_d^n$ denote the set of all multi-indices $\tilde I = (i_0, i_1,\ldots,i_n)\in\bbZ_{\ge 0}^{n+1}$ with $|\tilde I|= d$. Then all monomials 
$$
\{P_{m,\tilde I}\mid \tilde I\in\tilde{\mathcal I}_d^n, 0\le m\le n\}
$$
form a local basis at any point of $\mathrm{End}_d(\Pb^n)$.

Each map $F\in\cP_d^n$ can be naturally viewed as an element of $\mathrm{End}_d(\Pb^n)$ by the standard extension to $\tilde F\colon\Pb^n\to\Pb^n$ via homogenization. Then for each $m\neq 0$, a perturbation $\tilde F+tP_{m,\tilde I}$ can be viewed in the affine chart as $F+tP_{m,I}$ for the corresponding multi-index $I$. If $m=0$, and 
$$
\tilde F([z_0\colon\dots\colon z_n]) = [z_0^d\colon f_1(z_0,\bfz)\colon\dots\colon f_n(z_0,\bfz)],
$$
then the perturbation $\tilde F+tP_{0,\tilde I}$ is represented in the affine chart as
$$
F_t(\bfz) = \left(\frac{f_1(1,\bfz)}{1+t\bfz^I},\,\ldots,\, \frac{f_n(1,\bfz)}{1+t\bfz^I} \right).
$$

If a periodic point $\tilde\bfw_0$ of $\tilde F$ is finite (i.e., can be covered by the affine chart), then by passing to the affine chart, one can define the eigenvalue functions $\lambda$ and the maps $\rho_{k,\tilde\bfw_0}$, $k=1,\ldots,n$ on a neighborhood of $\tilde F$ in $\mathrm{End}_d(\Pb^n)$ in exactly the same way as in the polynomial case (c.f., Section~\ref{sec_diag_entries}). Similarly, one can use the affine chart to define the partial derivative operator $\partial_{m,\tilde I}$ on these functions for every $\tilde I\in\tilde{\mathcal I}_d^n$ and $m=0,\ldots,n$.

Note that the map $F_0$ extends via homogenization to a degree $d$ endomorphism $\tilde F_0$ of $\Pb^n$:
$$
\tilde F_0([z_0\colon z_1\colon\ldots\colon z_n]) = [z_0^d\colon z_1^d\colon\ldots\colon z_n^d].
$$

Next, we state a version of Lemma~\ref{computation_lemma} for the space $\mathrm{End}_d(\Pb^n)$.

\begin{lemma}\label{compute_Pn_lemma}
	Assume that in the above notation, none of the coordinates of a periodic point $\tilde\bfw_0 = [1\colon w_1\colon \ldots\colon w_n]\in\Pb^n$ of $\tilde F_0$ are equal to zero. Let $p$ be the period of $\tilde\bfw_0$ and define $\bfw_0 := (w_1,\ldots,w_n)\in\bbC^n$. Then for any multi-index $\tilde I = (i_0,\ldots,i_n)$, any $k\in\{1,\ldots,n\}$ and any $m\in\{0,\ldots,n\}$, the following holds:
\begin{equation*}
	\displaystyle\partial_{m,\tilde I}\rho_{k,\tilde \bfw_0}(\tilde F_0) = \begin{cases}
		0 & \text{if }m\neq 0\text{ and }m\neq k \\
		(i_kd^{p-1}-d^p)\sum_{i=0}^{p-1} (\bfw_0^{I_k})^{d^i} w_k^{d^i(i_k-d)} & \text{if }m\neq 0\text{ and } m=k \\
		-i_kd^{p-1}\sum_{i=0}^{p-1}(\bfw_0^I)^{d^i} & \text{if }m= 0.
	\end{cases}		
\end{equation*}
	
\end{lemma}
\begin{proof}
	The first two lines in the above formula for the partial derivatives $\partial_{m,\tilde I}\rho_{k,\tilde \bfw_0}(\tilde F_0)$ comes directly from Lemma~\ref{computation_lemma}. For the third line, if $m=0$, then the corresponding perturbation of $\tilde F_0$ can be written in the affine chart as
	\begin{equation}\label{reduction_eq}
	F_{0,t}(\bfz) = \left(\frac{z_1^d}{1+t\bfz^I},\,\ldots,\, \frac{z_n^d}{1+t\bfz^I} \right) = (z_1^d-tz_1^d\bfz^I, z_2^d-tz_2^d\bfz^I, \ldots, z_n^d-tz_n^d\bfz^I) +o(t),
	\end{equation}
	where $o(t)$ denotes the higher order terms in $t$. Then, applying Lemma~\ref{computation_lemma} to the latter expression yields the formula for $\partial_{m,\tilde I}\rho_{k,\tilde \bfw_0}(\tilde F_0)$ in the case when $m=0$.	
\end{proof}

Similarly to the polynomial case, we say that a multi-index $\tilde I = (i_0,\ldots,i_n)$ is \textit{admissible} if $i_j\neq d$ for any $j=0,\ldots, n$. Below we state the new version of Lemma~\ref{degree_count_lemma}.

\begin{lemma}\label{degree_lemma2}
	Under the conditions of Lemma~\ref{compute_Pn_lemma}, assume that $p\ge 3$ and the multi-index $\tilde I$ is admissible. Then 
\begin{equation*}
	\displaystyle\partial_{m,\tilde I}\rho_{k,\tilde \bfw_0}(\tilde F_0) = w_k^{-d^{p-1}} Q_{m,k,\tilde I}(\bfw_0),		
\end{equation*}
	where $Q_{m,k,\tilde I}(z_1,\ldots,z_n)$ is a polynomial, such that
	\begin{enumerate}[(i)]
		\item $Q_{m,k,\tilde I}$ is identically zero if and only if either $0<m\neq k$ or $m=0$ and $i_k=0$;
		\item if $d=2$, $m=0$, $i_k=1$ and $j=k$, then
		$$
		\deg_{z_k}Q_{0,k,\tilde I} = d^p-d^{p-2}.
		$$
		In all remaining cases, we have
		$$
		\deg_{z_j}Q_{m,k,\tilde I} \le d^p-d^{p-1}.
		$$
		
	\end{enumerate}
	Furthermore, for any $k=1,\ldots,n$ and any two distinct pairs $(m, \tilde I)$ and $(m', \tilde I')$, where $\tilde I$ and $\tilde I'$ are admissible multi-indices and $0\le m, m'\le n$, the polynomials $Q_{m,k,\tilde I}$ and $Q_{m',k,\tilde I'}$ do not contain monomials that are proportional to each other.	
		
\end{lemma}

\begin{proof}
	The existence of the polynomial $Q_{m,k,\tilde I}$ for $m\neq 0$ is proven in Lemma~\ref{degree_count_lemma}. Now if $m=0$, then, using Lemma~\ref{compute_Pn_lemma} and the fact that $w_k^{d^p}=w_k$, we obtain that if $i_k\neq d-1$, then
	$$
	Q_{0,k,\tilde I}(\bfz) = -i_kd^{p-1}\sum_{i=0}^{p-1} (\bfz^{I_k})^{d^i} z_k^{i_kd^i+d^{p-1}},
	$$	
	and if $i_k=d-1$, then
	$$
	Q_{0,k,\tilde I}(\bfz) = -i_kd^{p-1}\left((\bfz^{I_k})^{d^{p-1}} z_k + \sum_{i=0}^{p-2} (\bfz^{I_k})^{d^i} z_k^{(d-1)d^i+d^{p-1}}  \right).
	$$	
	In both cases the corresponding degrees of the polynomials can be computed.

	To compare the components of the polynomials $Q_{m,k,\tilde I}$ and $Q_{m',k,\tilde I'}$, we first observe that if $m,m'>0$, then both polynomials are non-identically zero precisely when $m=m'=k$, in which case the statement follows from the corresponding statement of Lemma~\ref{degree_count_lemma}.
	
	Now we consider the remaining two cases:
	
	(1) If $m=m'=0$, then it follows from the above formulas for $Q_{0,k,\tilde I}$ that from each monomial of $Q_{m,k,\tilde I}$, one can uniquely express both $i_k$ and $I_k$, hence, also the multi-index $I$. The latter implies that two polynomials $Q_{0,k,\tilde I}$ and $Q_{0,k,\tilde I'}$ with $\tilde I\neq \tilde I'$ cannot have proportional monomials.
	
	(2) Finally, assume that $m=0$ and $m'\neq 0$. Then $m'=k$, since otherwise $Q_{m',k,\tilde I'}$ is identically zero, and one can see that the polynomials $Q_{0,k,\tilde I}$ and $Q_{k,k,\tilde I'}$ have no proportional terms by comparing the degrees of $z_k$ in each term of the polynomials (c.f., Proposition~4.4 from~\cite{gorbovickis-rat}).

\end{proof}

Finally, we state a new version of Proposition~\ref{rho_independence_prop}.
Similarly to the polynomial case, for a fixed pair of integers $d$ and $n$, let $\tilde{\mathcal I}$ denote the set of all admissible multi-indices $\tilde I$. Let
$$
\tilde I\colon \{1,2,\ldots,N_{d,n}\}\to \tilde{\mathcal I}
$$
be any fixed bijection (enumeration of all admissible multi-indices).

\begin{proposition}\label{rho_independence_prop2}
	Assume that $n,d\ge 2$ and consider any finite sequence of integers 
	$$
	p_{0,1},\ldots,p_{0,N_{d,n}},p_{1,1},\ldots,p_{1,N_{d,n}},\ldots,p_{n,1},\ldots,p_{n,N_{d,n}},
	$$
	such that for all pairs of $k\in[0,n]$ and $j\in[1,N_{d,n}]$ the following holds: $p_{k,j}\ge 5$ if $d=2$ and $n=2$, and $p_{k,j}\ge 4$ if $d\ge 3$ or $n\ge 3$. Then
	there exists a corresponding finite sequence  
	\begin{equation*}
		\tilde\bfw_{0,1},\ldots,\tilde\bfw_{0,N_{d,n}},\tilde\bfw_{1,1},\ldots,\tilde\bfw_{1,N_{d,n}},\ldots,\tilde\bfw_{n,1},\ldots,\tilde\bfw_{n,N_{d,n}},
	\end{equation*}
	of periodic points of $\tilde F_0$, belonging to distinct periodic orbits, and a choice of indices $k_1,\ldots,k_{N_{d,n}}\in [1,n]$, such that
	
	(i) for each $k\in\{0,\ldots, n\}$ and $j\in\{1,\ldots, N_{d,n}\}$, the period of $\tilde\bfw_{k,j}$ is~$p_{k,j}$;
	
	(ii) the maps $\rho_{k_1,\tilde\bfw_{0,1}}, \ldots, \rho_{k_{N_{d,n}},\tilde\bfw_{0,N_{d,n}}}, \rho_{1,\tilde\bfw_{1,1}},\ldots, \rho_{1,\tilde\bfw_{k,N_{d,n}}}, \ldots, \rho_{n,\tilde\bfw_{n,1}},\ldots, \rho_{n,\tilde\bfw_{n,N_{d,n}}}$ are locally independent at $\tilde F_0$ in $\mathrm{End}_d(\Pb^n)$.
\end{proposition}
\begin{proof}
	We will select each $k_j$ so that $\tilde I(j)$ has a nonzero entry in the $k_j$-th position. According to Lemma~\ref{compute_Pn_lemma}, this guarantees that the polynomials $\partial_{0,\tilde I(j)}\rho_{k_j,\tilde\bfw_{0,j}}(F_0)$ are not identically zero for all $j=1,\ldots,N_{d,n}$.
	
	The Jacobian matrix of the maps 
	$$
	\rho_{k_1,\tilde\bfw_{0,1}}, \ldots, \rho_{k_{N_{d,n}},\tilde\bfw_{0,N_{d,n}}}, \rho_{1,\tilde\bfw_{1,1}},\ldots, \rho_{1,\tilde\bfw_{k,N_{d,n}}}, \ldots, \rho_{n,\tilde\bfw_{n,1}},\ldots, \rho_{n,\tilde\bfw_{n,N_{d,n}}}
	$$
	is no longer block diagonal, which is why the proposition is formulated for the whole Jacobian matrix rather than for each block as in the polynomial case covered by Proposition~\ref{rho_independence_prop}. (The sub-matrix, obtained by removing the first $N_{d,n}$ rows and columns will be block diagonal and will coincide with the Jacobian matrix from the polynomial case.) Nevertheless, exactly the same proof as the one for Proposition~\ref{rho_independence_prop} works here as well with the exception that due to part \textit{(ii)} of Lemma~\ref{degree_lemma2}, for $d=2$, instead of the inequality from Lemma~\ref{ineq_lemma}, we have to use a weaker inequality
	$$
	pnN_{d,n} < (d^{p-1} - d^{[p/2]})(d^{p-1}-1)^{n-2}(d^{p-2}-1).
	$$
	One can check that the latter inequality holds for $d=2$, $n=2$, $p\ge 5$ and for $d=2$, $n\ge 3$, $p\ge 4$. (See also Remark~\ref{ineq_remark}.) We leave the remaining details to the reader.	
\end{proof}

\begin{proof}[Proof of Theorem~\ref{main_theorem_Pn}]
	For $n=1$, Theorem~\ref{main_theorem_Pn} is a special case of a more general result~\cite{gorbovickis-rat}. 	
	To complete the proof of Theorem~\ref{main_theorem_Pn} for $n\ge 2$, we observe that the analogue of Proposition~\ref{lambda_loc_indep_prop} holds in the projective case as well. To prove it, one can repeat the proof of Proposition~\ref{lambda_loc_indep_prop}, replacing Proposition~\ref{rho_independence_prop} by the corresponding Proposition~\ref{rho_independence_prop2}, and using~(\ref{reduction_eq}) before applying Lemma~\ref{coincide_lemma}. Then Theorem~\ref{main_theorem_Pn} follows directly from Theorem~\ref{main_theorem_3} exactly as discussed in Section~\ref{sec_strategy}.	
\end{proof}


\begin{proof}[Proof of Corollary~\ref{cor-Doyle_Silverman}]
	Due to Theorem~\ref{main_theorem_Pn}, it is enough to check that for $p\ge 4$, the number of $p$-cycles of a generic map from $\mathrm{End}_d(\bbP^n)$ is at least as large as the dimension of the moduli space $\mathrm{End}_d(\bbP^n)/\mathrm{PGL}_{n+1}(\bbC)$. The latter one is equal to $(n+1)N_{d,n}$. On the other hand, similarly to the proof of Lemma~\ref{choice_lemma}, one can estimate the number of $p$-cycles from below by $(d^p-d^{[p/2]})(d^p-1)^{n-1}/p$. Thus, in order to prove the corollary, it is enough to check that for all $p\ge 4$, $n\ge 1$ and $d\ge 2$, the inequality
	\begin{equation}\label{per4_ineq}
	p(n+1)N_{d,n} \le (d^p-d^{[p/2]})(d^p-1)^{n-1}
	\end{equation}
	holds.
	
	For $n=1$, inequality (\ref{per4_ineq}) transforms into
	$$
	2p(d-1) \le d^p-d^{[p/2]}
	$$
	which is easy to verify for $p\ge 4$ and $d\ge 2$.
	
	For $n\ge 2$ we have
	$$
	p(n+1)N_{d,n} < p(n+1)(d+1)^n.
	$$
	Then one can check that 
	$$
	p(d+1)\le d^p-d^{[p/2]}
	$$
	and 
	$$
	(n+1)(d+1)^{n-1}\le (3d+3)^{n-1}\le (d^p-1)^{n-1},
	$$
	for $p\ge 4$ and $d\ge 2$. Combining the last three inequalities yields~(\ref{per4_ineq}), and hence, completes the proof of Corollary~\ref{cor-Doyle_Silverman}.	
\end{proof}

\section{Consequences about the bifurcation measure and the critical height}\label{sec_coro}
The main results of Gauthier-Taflin-Vigny in \cite{gauthier2023sparsity} were established for endomorphisms of $\Pb^n$ in all dimensions $n\geq2$, but only in dimension $n=2$ for polynomial endomorphisms of $\Cb^n.$ The missing ingredient in higher dimensions in that case was the independence of multipliers. Our initial goal in the article was to address this gap and extend the results of \cite{gauthier2023sparsity} to $\Cb^n$ for all $n\geq3.$ Using Theorem \ref{main_theorem_1}, we can deduce Corollary \ref{cor-mu-interior} and Corollary \ref{cor-unif} from Section~3 and 4 and from Theorem~7.2 in \cite{gauthier2023sparsity} respectively. We now explain this in more detail, starting by recalling classical results.

\subsection{Dynamics of regular endomorphisms and Lyapunov exponents}
The goal here is to define the bifurcation measure on $\tilde{\mathcal P}_d^n$ and to state the Bedford-Jonsson formula on Lyapunov exponents obtained in \cite{bedford-jonsson}, mainly to be able to use \cite[Theorem~7.2]{gauthier2023sparsity}. For more details and precise references, we refer the reader to these articles.
 
Let $d\geq2,$ $n\geq2$ and let $f$ be a regular polynomial endomorphism of $\Cb^n.$ The \emph{Green function} of $f$ is defined by
$$G_f(z):=\lim_{k\to\infty}d^{-k}\log^+\|f^k(z)\|.$$
This is a non-negative continuous plurisubharmonic function on $\Cb^n.$ From this, the \emph{Green current} of $f$ is then given by $T_f:=\ddc G_f$ and its \emph{equilibrium measure} is $\mu_f:=T_f^n.$ We denote by $L(f)$ the sum of all the Lyapunov exponents of $\mu_f$ and by $J_n(f)$ its \emph{small Julia set}, i.e. $J_n(f):=\supp(\mu_f).$

As $f$ is regular, it extends as an endomorphism of $\Pb^n,$ for which the hyperplane $H_\infty$ at infinity is totally invariant. The restriction $f_{|H_\infty}$ also has an equilibrium measure and we denote by $L_\infty(f)$ the sum of its Lyapunov exponents. Bedford and Jonsson established a formula connecting $L(f),$ $L_\infty(f)$ and the integration of the Green function with respect to the critical measure of $f.$ More precisely, if $\crit_f$ denotes (the closure in $\Pb^n$ of) the critical set of $f$ in $\Cb^n$ then we  set
$$\mu_{f,\crit}:=T_f^{n-1}\wedge[\crit_f]\ \text{ and }\ G_{f,\crit}:=\int_{\Cb^n}G_f\mu_{f,\crit}.$$
Then, \cite[Theorem~3.2]{bedford-jonsson} stated that
\begin{equation}\label{eq-bj}
L(f)=\log d+L_\infty(f)+G_{f,\crit}.
\end{equation}
An important fact for us is that the three functions
$$
L\colon f\mapsto L(f),\text{ } L_\infty\colon f\mapsto L_\infty(f)\ \text{ and } G_{\crit}\colon f\mapsto G_{f,\crit}
$$
are plurisubharmonic \cite{ds-allupoly} and continuous \cite{BB1} on $\mathcal{P}_d^n$. Moreover, $L$ is invariant under holomorphic conjugacy thus if $[f]$ denotes the class of $f$ in $\tilde{\mathcal P}_d^n$ then $\tilde L\colon[f]\mapsto L(f)$ is well-defined on $\tilde{\mathcal P}_d^n.$ If for simplicity we denote by $N:=nN_{d,n}$ the dimension of $\tilde{\mathcal P}_d^n,$ then complex Monge-Amp\`ere $\mu_\bif:=(\ddc \tilde L)^N$ of $\tilde L$ is the \emph{bifurcation measure} introduced by Bassanelli-Berteloot in \cite{BB1}.

A technical difficulty with the moduli space $\tilde{\mathcal P}_d^n$ is that its elements are classes and several dynamical objects, such as the critical set, cannot be naturally associated to such classes. To overcome this, we consider the space $\hat{\mathcal P}_d^n\subset\mathcal P_d^n$ defined below \eqref{eq-monic} which has the advantage that the projection $\pi\colon\hat{\mathcal P}_d^n\to\tilde{\mathcal P}_d^n$ is a finite ramified cover. This holomorphic family of mappings gives a global endomorphism
\begin{align*}
F\colon\Pb^n\times\hat{\mathcal P}_d^n&\to\Pb^n\times\hat{\mathcal P}_d^n\\
(z,f)&\mapsto(f(z),f).
\end{align*} 
Since each $f\in\hat{\mathcal P}_d^n$ is a regular polynomial endomorphism of $\Cb^n,$ the critical set of $F$ decomposes as $H_\infty\times\hat{\mathcal P}_d^n\cup\crit_F$ where the fiber of $\crit_F$ above $f$ is exactly $\crit_f.$

The bifurcation measure in $\hat{\mathcal P}_d^n$ is simply $\hat\mu_\bif:=(\ddc \hat L)^N,$ where $\hat L=L_{|\hat{\mathcal P}_d^n},$ which satisfies $\hat\mu_\bif=\pi^*\mu_\bif$. We also defined $\hat\mu_\bif^{\mathrm{pol}}:=(\ddc \hat G_{\crit})^{N},$ where $\hat G_{\crit}:=(G_{\crit})_{|\hat\mu_\bif}$. The equation \eqref{eq-bj} and the fact that $L_\infty$ is plurisubharmonic ensure that
\begin{equation}\label{eq-measures}
\hat\mu_\bif\geq\hat\mu_\bif^{\mathrm{pol}}.
\end{equation}
We will apply \cite[Theorem~7.2]{gauthier2023sparsity} when $S=\hat{\mathcal P}_d^n$, $f=F$ and $\mathcal Y=\crit_F.$ The main assumption then corresponds to 
\begin{equation}\label{eq-assum}
\int_{\hat{\mathcal P}_d^n} \hat G_{\crit}\hat\mu_\bif^{\mathrm{pol}}>0.
\end{equation}
As we will see, on the open set $\hat\Omega$ where we will check this positivity, the inequality \eqref{eq-measures} is actually an equality.

\subsection{The support of the bifurcation measure has non-empty interior}
Here, we prove Corollary \ref{cor-mu-interior}. Roughly speaking, the approach of Gauthier-Taflin-Vigny to obtain a non-empty open subset in the support of the bifurcation measure in $\tilde{\mathcal P}_d^n$ is the following. 
\begin{enumerate}
\item\label{it-con} Construct a robust heterodimensional cycle between a repelling hyperbolic set $\Lambda$ and a saddle fixed point $p$. More precisely, there exists a connected non-empty open subset $\Omega$ of $\mathcal P_d^n$ such that each $f\in\Omega$ has a repelling hyperbolic set $\Lambda(f)$ and a saddle fixed point $p(f)$ with  $W^s_{p(f)}\cap W^u_{\Lambda(f)}\neq\varnothing$ and $W^u_{p(f)}\cap\Lambda(f)\neq\varnothing$. The first condition is easily achieved robustly as $W^u_{\Lambda(f)}$ is an open subset of $\Cb^n$,  while the latter requires more care and \cite{gauthier2023sparsity} used a mechanism called blender to obtain it. Moreover, several other technical assumptions on $\Omega\subset\mathcal P_d^n$ are required (see  \cite[Section~3]{gauthier2023sparsity}).
\item\label{it-conj} If the projection $\tilde\Omega$ of $\Omega$ in $\tilde{\mathcal P}_d^n$ does not intersect the interior of the support of the bifurcation measure then $\tilde\Omega$ contains in a dense way submanifolds $M$ of positive dimension with a special property denoted by $(\star)$ which implies that the eigenvalue functions associated to periodic points on the small Julia set are constant on $M.$
\item\label{it-mul} In particular, infinitely many multipliers are constant on such submanifolds, contradicting the independence of multipliers.
\end{enumerate}
The point \eqref{it-con} is given by \cite[Theorem~4.1]{gauthier2023sparsity} which is also available in $\mathcal P_d^n$ with $n\geq3$. The point \eqref{it-conj} corresponds to the proof of Theorem~C in \cite{gauthier2023sparsity} using Theorem~3.4 there. The only difference in our setting is that we need information on the multipliers (i.e. eigenvalues of $D_xf^p$) while  \cite[Theorem~3.4]{gauthier2023sparsity} only gives information on their product (i.e. the determinant of $D_xf^p$). However, the proof of Theorem~3.4 actually establishes that if $M$ satisfies $(\star)$ then all $f,f'\in M$ are holomorphically conjugated in neighborhoods of their small Julia sets $J_n(f)$ and $J_n(f').$ Hence, the associated eigenvalue functions are constant on $M.$ Finally, the contradiction in the point \eqref{it-mul} follows from Theorem \ref{main_theorem_1}, which ensures that submanifolds $M$ satisfying $(\star)$ cannot be dense in $\tilde\Omega.$
\subsection{Uniform control on the critical preperiodic points}
To prove Corollary \ref{cor-unif}, it suffices to verify that the assumption \eqref{eq-assum} of \cite[Theorem~7.2]{gauthier2023sparsity} holds, and then use a density argument to extend the result, which is originally stated over $\overline{\Qb}$, to maps $f$ defined over $\Cb$.

The first step is immediate after the observation that the open set $\Omega\subset\mathcal P_d^n$ used above is a small neighbordhood of a map
$$f_0\colon (z,w,y_3,\ldots,y_n)\mapsto (\alpha z+\epsilon w+\beta zw+w\sum_{i=3}^n\tau_iy_i, a(w^2-1),\sigma_3y_3,\ldots,\sigma_ny_n)+c(z^d,w^d,y_3^d,\ldots,y_n^d),$$
where $\alpha,\beta,\epsilon,\tau_i,\sigma_i,a,c$ are constants with $0<|c|<1$ small and $|a|$ can be chosen arbitrarily large with respect to $|\beta|$ and $|\tau_i|$ (see \cite[Section~4.5]{gauthier2023sparsity}). The dynamics at infinity of such a map is conjugated to the power map if $d\geq3$. And a simple computation gives that it is arbitrarily close to it when $d=2$ for $|\beta/(a+c)|$ and $|\tau/(a+c)|$ small. Therefore, by reducing $\Omega$ if necessary, we can assume that the restriction of $f\in\Omega$ to the hyperplane at infinity $H_\infty$ is hyperbolic. In particular, $f\mapsto L_\infty(f)$ is pluriharmonic on $\Omega$ and thus, $\ddc L=\ddc G_{\crit}$ on the image $\tilde\Omega$ of $\Omega$ in $\tilde{\mathcal P}_d^n.$ In particular, if $\hat\Omega:=\pi^{-1}(\tilde\Omega)$ then $\hat\mu_\bif=\hat\mu_\bif^{\mathrm{pol}}$ on $\hat\Omega$ in $\hat{\mathcal P}_d^n.$ Since we have $\tilde\Omega\subset\supp(\mu_\bif),$ $\hat\mu_\bif=\pi^*\mu_\bif,$ $\hat G_{\crit}\geq0$ and $\hat\mu_\bif^{\mathrm{pol}}=(\ddc \hat G_\crit)^N$, we must have $\hat\Omega\subset\supp(\hat\mu_\bif^{\mathrm{pol}})$ and
$$\int_{\hat{\mathcal P}_d^n}\hat G_{\crit}\hat\mu_\bif^{\mathrm{pol}}\geq\int_{\hat\Omega} \hat G_{\crit}\hat\mu_\bif^{\mathrm{pol}}>0$$
which corresponds exactly to \eqref{eq-assum}. Hence, \cite[Theorem~7.2]{gauthier2023sparsity} gives a non-empty Zariski open subset $\hat U$ of $\hat{\mathcal P}_d^n$ with a uniform control, for $f\in\hat U(\overline\Qb),$ on the points of small canonical height lying on the critical set of $f.$ As preperiodic points have zero height, the result applies to them. More precisely, there exists $B\geq1$ such that if $f\in\hat U(\overline\Qb)$ then there is an algebraic subset $W_f$ of $\Cb^n$ such that $\deg(W_f)\leq B$ and
$$\mathrm{Preper}(f)\cap\crit_f\subset W_f.$$
If we could extend the result to $\hat U(\Cb)$ then the set $U$ in Corollary \ref{cor-unif} would simply be the preimage of $\pi(\hat U)$ by the projection $\mathcal P_d^n\to\tilde{\mathcal P}_d^n,$ i.e. $U$ is the set of maps $f$ whose conjugacy class intersects $\hat U.$

Now, it remains to extend the result from $\hat U(\overline\Qb)$ to $\hat U(\Cb)$, which simply follows from the compactness of the set of codimension $2$ algebraic subsets of $\Cb^n$ of degree bounded by $B.$ More precisely, assume that some $f\in \hat U(\Cb)$ does not satisfy the statement, i.e. the set $\mathrm{Preper}(f)\cap\crit_f$ is not contained in any codimension $2$ algebraic subvariety of $\Cb^n$ of degree bounded by $B.$ Hence, by compactness there exist $p>q\geq0$ such that the same holds for $\mathrm{Preper}_{p,q}(f)\cap\crit_f$ where $\mathrm{Preper}_{p,q}(f):=\{z\in\Pb^n\,;\, f^p(z)=f^q(z)\}.$ On the other hand, both $\crit_F$ and $\mathrm{Preper}_{p,q}:=\{(z,g)\in\Pb^n\times\hat{\mathcal P}_d^n\,;\, g^p(z)=g^q(z)\}$ are defined over $\overline\Qb.$ A difficulty here is that their intersection $Z:=\crit_F\cap\mathrm{Preper}_{p,q}$ might have arbitrarily many irreducible components with respect to $p$ and $q$ and it could be that it is not of pure dimension. However, if $\rho\colon\Pb^n\times\hat{\mathcal P}_d^n\to\hat{\mathcal P}_d^n$ is the natural projection then we can consider all the irreducible components $(Z_i)_{i\in I}$ of $Z$ such that $f\in\rho(Z_i)$ for each $i\in I.$ The set $Y:=\cap_{i\in I}\rho(Z_i)$ is also defined over $\overline\Qb$ so there exists a sequence $(f_n)$ in $(Y\cap \hat U)(\overline\Qb)$ converging to $f.$ The Remmert open mapping theorem (see e.g. \cite[Chapter V.6]{loja}) applied to each $\rho_{|Z_i}\colon Z_i\to\rho(Z_i)$ ensures that each point in $\mathrm{Preper}_{p,q}(f)\cap\crit_f$ is a limit of points in $\mathrm{Preper}_{p,q}(f_n)\cap\crit_{f_n}.$ Thus, $\mathrm{Preper}_{p,q}(f)\cap\crit_f$ is contained in any limit value of $(W_{f_n})$ which gives the desired contradiction.

\bibliographystyle{amsalpha}
\bibliography{biblio}
\end{document}